\documentclass[oneside,english,10pt]{amsart}
\usepackage{color}
\usepackage{amsthm}
\usepackage{amstext}
\usepackage{graphicx}
\usepackage{amscd}
\usepackage{amsfonts}
\usepackage[colorlinks,linkcolor=blue, pdfstartview=FitH]{hyperref}
\usepackage{amssymb}
\usepackage{esint}
\usepackage{enumerate}
\usepackage{hyperref}
\usepackage{amsmath, amssymb, amsthm, mathrsfs, graphicx}
\usepackage{hyperref, enumitem, xspace, ifthen,comment}
\usepackage[all]{xy}
\usepackage{tikz}
\usetikzlibrary{arrows,shapes,trees}

\hypersetup{citecolor=red}
\begin{document}

\title[\null]
{ On Fano threefolds with semi-free ${\mathbb C}^*$-actions, I }

\author[\null]{Qilin Yang and Dan Zaffran}

\address{ Department of Mathematics,
Sun Yat-Sen University, ~510275, ~ Guangzhou, ~P. R.  CHINA.} %
\email{yqil@mail.sysu.edu.cn}

\address{Department of Mathematical Science,
Dan Dahle Building, 101,
Florida Institute of Technology
150 W. University Blvd
Melbourne, Florida 32901, U. S. A.} %
\email{dzaffran@fit.edu}


\renewcommand{\labelenumi}{(\Alph{enumi})}
\newcommand{\C}{\mathbb C}
\newcommand{\F}{\mathbb F}
\newcommand{\Pj}{\mathbb P}
\newcommand{\Z}{\mathbb Z}
\newcommand{\R}{\mathbb R}
\newcommand{\fk}{\mathfrak{k}}
\newcommand{\fm}{\mathfrak{m}}
\newcommand{\cf}{\mathrm{cf.}}
\newcommand\grad{\rm grad}
\newcommand\lgr{\longrightarrow}
\newcommand\rw{\rightarrow}
\newcommand\lmp{\longmapsto}
\newcommand\lie{{\mathscr L}{\rm ie}}
\newcommand\Lm{\Lambda}
\newcommand\lmd{\lambda}
\newcommand\al{\alpha}
\newcommand\dt{\delta}
\newcommand\G{\Gamma}
\newcommand\eg{e.g.}
\newcommand\ie{i.e.}
\newcommand\resp{\rm resp.}
\newcommand\Hom{\rm Hom}
\newcommand\ov{\overline}
\newcommand\bt{\beta}
\newcommand\om{\omega}
\newcommand\df{\rm d}
\newcommand\bi{\bar{i}}
\newcommand\bj{\bar{j}}
\newcommand\bl{\bar{l}}
\newcommand\bk{\bar{k}}
\newcommand\bp{\bar{p}}
\newcommand\gd{\rm grad}
\newcommand\lge{\langle}
\newcommand\rg{\rangle}

\newtheorem{theorem}{Theorem}[section]
\newtheorem{proposition}[theorem]{Proposition}
\newtheorem{corollary}[theorem]{Corollary}
\newtheorem{lemma}[theorem]{Lemma}
\newtheorem{example}[theorem]{Example}
\newtheorem{remark}[theorem]{Remark}
\newtheorem{definition}[theorem]{Definition}

\begin{abstract}
Let $X$ be a Fano threefold and $\C ^* \times X\rightarrow X$ an algebraic action. Fix a maximal compact subgroup $S^1$ of ${\mathbb C}^*.$ Then
$X$ has
a $S^1$-invariant K\"ahler structure and the corresponding  $S^1$-action
admits an equivariant moment map which is at the same time a perfect Bott-Morse function. We will initiate a program to classify the Fano threefolds with semifree ${\mathbb C}^*$-actions using the Morse theory and  holomorphic Lefschetz fixed point formula as the main tools. In this paper we give a complete list
of all possible  Fano threefolds without the interior isolated fixed points for any  semifree ${\mathbb C}^*$-actions. Moreover when the actions whose fixed point sets have only two connected components and a few of the rest cases, we give the realizations of  the semifree $\C^*$-actions.
\end{abstract}
\maketitle
\baselineskip=10pt

\section{introduction}
Orlik and Wagreich gave a complete classification of smooth projective algebraic surfaces with $\C^*$-actions in  \cite{ow77}.
Building on work of  Audin, in \cite{ka} Karshon gave a complete classification of Hamiltonian circle actions on compact $4$-manifolds, proved that each of those  $4$-manifolds admits a compatible, invariant K\"ahler structure. More precisely,  compact symplectic $4$-manifolds with Hamiltonian circle actions are in fact obtained by equivariant symplectic blow ups at fixed points of the complex ruled surfaces. This classification is highly similar to  the projective algebraic surfaces with $\C^*$-actions  in  \cite{ow77}.
The followed question is to give a  classification of Hamiltonian circle actions on compact $6$-manifolds. In \cite{lh} Li decided the possible types of fixed point set of a semifree Hamiltonian circle action on the $6$-manifold with the second Betti number less than $3.$
 The general classification  is  difficult since we lack a complete classification of smooth, simply-connected, compact $6$-manifolds. An easier question is the classification of the projective algebraic $3$-folds (they are real $6$ dimensional complex algebraic varieties) with Hamiltonian circle actions.
The Mori theory of the extremal ray contractions is very powerful in the study of the
algebraic $3$-folds. Using this theory, Mori and Mukai\cite{mm1} gave a complete classification of Fano $3$-folds, building on Iskovskih's work about Fano $3$-folds of Picard number one.

Note that Fano manifolds are uniruled. A classification  of Fano $3$-folds with $\C^*$ actions is undoubtedly helpful in studying the  Hamiltonian circle actions on compact $6$-manifolds. Motivated by this crude observation,  we wish to deal with the simplest case-the classification of Fano $3$-folds with {\it semifree} $\C^*$-actions, in this and a following paper.

At the beginning of this work, we want to imitate the procedure  in  \cite{ow77} and \cite{ka} to give a classification of Fano $3$-folds with $\C^*$-actions. We had to give up after some attempt, mainly because of the complexity
 of birational surgery in higher dimension, partially due to the authors have no enough knowledge of the computational details in the paper \cite{mm1}.
In this paper we go a different way and use directly the classification results established by Iskovskih, Mori and Mukai. Suppose a Fano $3$-fold $X$ admits  a semifree $\C^*$-action, we could label its fixed point data by different types. Use the Morse theory for the moment map as established in \cite{fr59},\cite{cs79},\cite{at82},
we could read the Hodge numbers from the assumed types of fixed point data. Then use the holomorphic Lefschetz fixed point formula and  Atiyah-Bott- Berline-Vergne
localization theorem,  we could get more global numerical  invariants of $X$ from its local fixed point data. These numerical characters will greatly narrow the range the Fano manifolds which have $\C^*$ actions, and as a result we can study them case by case and finally get a complete classification.

 The holomorphic Lefschetz fixed point formula and  Atiyah-Bott-Berline-Vergne
localization theorem combine the local information near the fixed submanifolds with the global topological and analytical invariants of $X.$ The difficult is to decide the  local information near the fixed submanifolds.
The following observations are starting point of our work.
Assume $X$ is a Fano manifold with an effective semifree  $\C^*$-action.
Choose a  $S^1$-invariant K\"ahler structure  on $X,$
since $X$ is simply connected we may further assume that $S^1$-action is Hamiltonian. Then it admits
an equivariant moment map $J:X\rw \R^1$ which is  a perfect Bott-Morse function \cite{fr59},\cite{at82}.
The critical points of $J$ are exactly
the fixed points sets $X^{S^1}=X^{\C^*}.$ The connected components of  $X^{S^1}$ are K\"ahler
 hence algebraic submanifolds
of $X.$
Let $r \in {\rm im}(J)$; then $J^{-1}(r)/S^1 = X_r$ is called the Marsden-Weinstein reduced space at $r.$ If
the action is semi-free, then, for any regular value $r,$ the reduced space $X_r$ is a smooth manifold; for a
singular value $r,$ for $X$ of arbitrary dimension it is not clear whether or not $X_r$ is a smooth manifold. However,
for a Fano $3$-fold $X$(it has real dimension $6$), the fixed point sets are the minimum, the maximum or of index $2$
or co-index 2. By \cite{gs89}, all the reduced spaces (including those on critical levels)
are smooth manifolds. Moreover, when the index $2$ and co-index $2$ fixed point sets
consist of surfaces, the reduced spaces are all diffeomorphic.  The natural diffeomorphism between the reduced spaces is
induced by the gradient flow of the Morse function, hence there is a well defined diffeomorphism between the reduced spaces.
 Note the circle action is the restriction of the holomorphic $\C^*$-action, the Marsden-Weinstein reduced spaces $X_a$ are  K\"ahler manifolds by \cite{hl94}.
 Since $X$ is a Fano $3$-fold,  $X_a$ are also Fano surfaces by \cite{fut}. Therefore all $X_a$ are del Pezzo surfaces. As a result, if there are index $2$ and co-index $2$ isolated fixed points, the transformations of diffeomorphic type of
the reduced spaces are  compositions of blow ups or blow downs between del Pezzo surfaces.
  In section \ref{sect5}, we  prove that on a
pointwise fixed complex algebraic surface  the moment map will take maximum or minimum; and if $C$ is a pointwise fixed complex algebraic curve
   on it  the moment map takes maximum or minimum, then  $C$ is a rational curve. These observations provide enough information for the computations of the contributions of fixed  submanifolds in the holomorphic Lefschetz fixed point formula and  Atiyah-Bott-Berline-Vergne
localization formula, and make it possible for us to study the classification problem case by case since Fano manifolds of complex dimension no more than two have very few possibilities.

The fixed points where the moment map does not take maximum or minimum are called interior fixed points in this paper. In section \ref{sect3} the connected components of  interior fixed points are classified into three types: fixed curves, isolated fixed point of isotropic weight  type $(1,2)$, isolated fixed point of isotropic weight type $(2,1).$ In this paper we will give a classification of  Fano $3$-folds with semifree $\C^*$-actions which have no interior isolated fixed points. In a following paper we will deal with the rest cases, i.e., the actions have at least one interior isolated fixed point.

\begin{theorem} Let $X$ be a Fano $3$-fold equipped with an effective, semifree, and algebraic $\C^*$-action. Suppose that the fixed point set $X^{\C^*}$ has no  no interior isolated fixed points. If $X^{\C^*}$  has two connected components, then $X$ is listed in Table $5;$ if $X^{\C^*}$  has at least three connected components, then $X$ is $P^1\times P^1\times P^1,~~~ P^1\times\widetilde{P^2}$ or
 a possible one in the rest cases in Table $6.$
\end{theorem}

The paper is organized as follows. In Section \ref{sect2}  we  introduce some basic notions of algebraic
 transformation group, and display the relations among ruled manifolds, Fano manifolds and algebraic manifolds with $\C^*$-actions.
In Section \ref{sect3}, using the isotropic representation we give a classification of the types of the fixed point submanifolds  of algebraic threefolds with semifree $\C^*$-action.
In Section \ref{sect4}, we prove the  homology  and cohomology groups with integer coefficient of a Fano $3$-fold $X$ with $\C^*$-actions are torsion free and the fixed surfaces are del Pezzo surfaces. We give the expression of the Hodge number of $X$ form those of its fixed submanifolds.
In Section \ref{sect5}, we compute the contributions of the fixed submanifolds of different types  to the  holomorphic Lefschetz fixed point formula and  Atiyah-Bott-Berline-Vergne
localization formula. In Section \ref{sect6}, we give a classification of Fano $3$-folds with semifree $\C^*$-actions whose fixed point set has only two connected components. Finally in Section \ref{sect7}, we give a classification of Fano $3$-folds with semifree $\C^*$-actions whose interior fixed point set consists of algebraic curves.

\section{ $\C^*$ actions on  general projective $3$-folds}\label{sect2}
\setcounter{section}{2} \setcounter{equation}{0}

Let $G$  be an algebraic group and $X$ an algebraic variety. A group action  $\sigma:G\times X\rw X$ is  called an algebraic action if $\sigma$ is a morphism between algebraic varieties.
 Usually we denote $\sigma(g,x)=g\cdot x$ for brevity. Then
 $g\cdot(h\cdot x)=(gh)\cdot x$ and $1\cdot x=x$
  for any $g,h\in G$ and $x\in X.$ Here $``1"$ denotes the unit of $G.$
 The action is called {\it effective} if $g\cdot x=x$  for all $x\in X$
then $g=1.$ Throughout this paper we  only study effective group actions. For any $x\in X,$ the closed subgroup $G_x=\{g\in G|g\cdot x=x \}$ is called the isotropy group at $x$ and there is a naturally induced linear representation of $G_x$ on the tangent space $T_x X,$ called {\it isotropic representation}. Let $\C^*$ be the multiplicative group of complex number field $\C,$ it is the simplest reductive algebraic group.
In this paper we will only consider algebraic $\C^*$-actions on complex algebraic varieties if without special explanations.

A $\C^*$-action on an complex affine variety $V={\rm Spec} (A)$ is equivalent to a grading of the
ring $A.$ If $A$ is graded ring
we can find homogeneous generators of the $\C^*$-algebra $A,$ say $x_1, \cdots, x_n.$ The homomorphism
$\phi: \C [X_1,\cdots, X_n]\rw A$ defined by $\phi(X_j)=x_j$ defines an embedding of $X$ in $\C^n.$ Let $m_j$ be the
degree $x_j.$ Then define an action of $\C^*$ on $\C^n$ by
\begin{equation} t\cdot (z_1,\cdots,z_n)=(t^{m_1}z_1,\cdots,t^{m_n}z_n).\label{action}\end{equation}
This action on $\C^n$ leaves $V$ invariant. The following Lemma was also used in \cite{ow77}.

\begin{lemma}\label{affine} Suppose that $\C^*$ acts effectively on an affine variety $V.$ Then there exists an affine variety
$W$ and a Zariski open subset $U$ of $V$
such that $U$ is equivariantly isomorphic to $W\times \C^*,$ where $\C^*$ acts on $W\times \C^*$ by translations
on the second factor.
\end{lemma}
\begin{proof} Firstly assume that $V=\C^n$ is a complex vector space, and the action is given
as (\ref{action}). In this case, we take
$$U=\{(z_1,\cdots,z_n)|z_1\not=0,\cdots,z_n\not=0\}={\rm Spec}(\C[z_1,\cdots,z_n,z_1^{-1},\cdots,z_n^{-1}]).$$
Define a grading of
$A=\C[z_1,\cdots,z_n,z_1^{-1},\cdots,z_n^{-1}]$
 by degree $z_i=m_i.$
Let $R$ be the subalgebra
consists of elements of degree zero and $W={\rm Spec} R$. Since the action is effective, ${\rm g.c.d}(m_1,\cdots,m_n)=1.$ Hence there exist
integers $a_1,\cdots,a_n$ such that $a_1m_1+\cdots+a_n m_n=1.$ Let $Z=z_1^{a_1}\cdots z_n^{a_n},$ then
$A=R[Z,Z^{-1}],$ hence $U=W\times \C^*$ in this case.

In the case of general affine variety $V$, by the embedding theorem of Sumihiro \cite{sum74}, there is a Zariski open subset which
is equivariantly embedded in $\C^n.$ So we may assume $V\subset \C^n$ is an affine subvariety cut out by an ideal
$I\subset R[Z,Z^{-1}].$
Naturally it has grading, and $I=\oplus_{j\in\Z} I_j=\oplus_{j\in \Z}I_0 Z^j.$ Let $W={\rm Spec}(R/{I_0}),$ then
$V=W\times \C^*.$
\end{proof}

\begin{theorem}\label{rule} Suppose that $X$ is nonsingular 3-fold
 equipped with an algebraic
$\C^*$-action. Then there exist a smooth projecitve surface $S$, a 3-fold $Y$ with an algebraic $\C^*$ action, and equivariant morphisms $\varphi_1,\varphi_2,$
$$
\xymatrix {Y\ar[d]^{\varphi_1}\ar[dr]^{\varphi_2}\\
X\ar@{.>}[r]^f & S\times P^1}
$$
such that
\begin{enumerate}
\item
the $\C^*$-action on $S\times P^1$ is given by $t\cdot (y,[z_0:z_1])=(y,[tz_0,z_1]);$

\item
$\varphi_1$ and $\varphi_2$ are equivariant;

\item
$\varphi_1$ is composed of blow ups whose
centers are the fixed points or invariant curves of the $\C^*$-action.
\end{enumerate}
\end{theorem}

\begin{proof}
By Lemma \ref{affine}, there is an open invariant subset $U\subset X$ equivariantly isomorphic
to $S_0\times \C^*,$ where $S_0$ is a smooth surface. Therefore there is
an open embedding $f:U\rw S\times P^1,$ where $S$ is a
smooth completion of $S_0.$ Hence $f$
defines a dominant rational map $f:X\dashrightarrow S\times P^1$ which is equivariant relative to the $\C^*$ action. The set $I(f)$ of points of indeterminacy of $f$
is an algebraic subset of $X$ of codimension at least two, hence $I(f)$ is the union  of finite
number  of points  and finite number of curves. If $x\in I(f)$ is an isolated point then it is a fixed point,
otherwise $f$ is defined at $x$ by formula $f(x)=t^{-1}\cdot f(t\cdot x).$ If $x\in C\subset I(f)$ where
$C$ is curve, by the same reason that $t\cdot x\in C,$ hence $C$ is invariant. Note that the singular points
are also fixed, therefore we may assume that $C$ is smooth.
Let $X_0=X.$ Define inductively $f_i:X_i\rw X_{i-1}$ to be the blow ups along the fixed points and
invariant curves located in the indeterminacy  of $f\circ f_1\circ\cdots\circ f_{i-1}.$
By Hironaka's theorem on resolution of the indeterminacies of rational maps \cite{hir64}, there is an integer $k$ such that $X_k$ is smooth
and $f\circ f_1\circ\cdots\circ f_k$ is defined on $X_k.$ Let $Y=X_k,$
$\varphi_1=f_1\circ\cdots\circ f_k$ and $\varphi_2=f\circ\varphi_{1}.$
Note $f_i$ and $f$ are equivariant and hence so are
$\varphi_1$ and $\varphi_2.$
\end{proof}

\begin{proposition}\label{rul}Let $X$ be nonsingular projective $3$-fold which admits  an effective $\C^*$-action, then there is a rational curve   though a general point of $X.$
\end{proposition}
 \begin{proof}
By Chevalley's theorem, the orbit  $\C^*\cdot x$ at any point $x\in X$ is a constructible set, hence the closure $\overline{\C\cdot x}$ is an algebraic subset.  Now assume $x$ is not a fixed point. Then $\overline{\C\cdot x}-\C\cdot x$ is a $\C^*$-invariant subset consists of lower dimensional orbits, therefore the algebraic subset $\overline{\C\cdot x}-\C\cdot x$ is a finite set  and it is pointwise fixed by $\C^*$-action.  We claim that the limits $\lim_{t\rw 0}t\cdot x$ and $\lim_{t\rw \infty }t\cdot x$ must exist. In fact, the action morphism $\C^*\times X\rw X$ yields an algebra homomorphism:
$$\C[X]\rw \C[\C^*\times X]=\C[\C^*]\otimes \C[X],$$
for $f\in \C[X],$ we may write
$$f(t\cdot x)=\sum_{i=1}^m\phi_i(t)\psi_i(x),$$
where $\phi_1,\cdots,\phi_m\in \C[\C^*]$ and $\psi_1,\cdots,\psi_m\in \C[X].$ If one of the limits $\lim_{t\rw 0}t\cdot x$ and $\lim_{t\rw \infty}t\cdot x$ does not exist, then there exist series $\{t^1_i\},\{t^2_i\}\subseteq \C^*$ with $\lim_{i\rw \infty}t^1_i=\lim_{i\rw \infty}t^2_i,$ but the limits $\lim_{i\rw \infty}t^1_i\cdot x$ and $\lim_{i\rw \infty}t^2_i\cdot x$ are different.
Since $\overline{\C\cdot x}-\C\cdot x$ is a finite set, there exist $\epsilon_0>0$ such that
$$\lim_{i\rw\infty}|f(t^1_i\cdot x)-f(t^2_i\cdot x)|\geq \epsilon_0,$$
but$$\lim_{i\rw\infty}|\sum_{k=1}^m(\phi_k(t^1_i)\psi_i(x)-\phi_k(t^1_i)\psi_i(x))|=0,$$ this is a contradiction.

Hence $\overline{\C\cdot x}$ is
 isomorphic to a rational curve if $x$ is not a fixed point.  Since the fixed point set $X^{\C^*}$ is an algebraic subset of $X$
and the action is effective, $X^{\C^*}$ is a proper algebraic subset of $X.$ Hence there is a rational curve passing through a generic point of $X.$
\end{proof}

Recall that a projective variety $X$ of dimension $n$ is called {\it ruled} ({\rm resp.} {\it uniruled}) if there is a variety $Y$ of dimension
$n-1$ and a birational ({\rm resp.} {\it dominant rational}) map $Y\times P^1\dashrightarrow X.$
 By definition a projective variety is uniruled then there is a rational curve through a general point. If $X$ is irreducible and  $X-X^{\C^*}$ is a finite union of $\C^*$-invariant affine open subsets, by Luna's slice theorem, each invariant affine open subset is ruled and hence $X$ is ruled.
 We conjecture that any nonsingular projective variety equipped with  an effective $\C^*$-action is ruled.

A projective manifold $X$ with ample canonical line bundle
$K_X$ is of general type, it admits no effective $\C^*$
actions since its automorphism group is a finite group. If the  anticanonical line bundle
of $X$  is ample, it
is called a {\it Fano} manifold.  A projective variety is called {\it rationally connected}  if there is rational curve through a general pair of points.
Kollar, Miyaoka, Mori   \cite{kmm92} and Campana \cite{cm} proved that nonsingular Fano varieties are rational connected, in particular
they are uniruled.
However not every Fano manifold admits an effective $\C^*$-action. Fano surfaces have the following
complete lists: (1) $P^2;$ (2) Analytic $P^1$ bundles over $P^1;$
(3) (Denoted by $S_n$) The blow up of $P^2$ at $9-n$ points ($1\leq n \leq 7$) in general positions.
Clearly the first two classes and  $S_7$ in the case (3) admit effective $\C^*$
actions. However when $n\leq 6,$
the automorphism groups ${\rm Aut}_{\mathscr O}(S_n)=1,$
in particular $S_n$ have no nontrivial $\C^*$-actions when $n\leq 6.$

It is an interesting problem to classify those Fano manifolds which admit effective $\C^*$-actions.
Since on the one hand it is very hard to classify all fano manifolds as the dimensions increasing, those
have continuous symmetry naturally deserve more attentions;
 on the
other hand the study of $\C^*$-actions on Fano manifolds give us numerous examples which are urgently
needed in the study of Hamiltonian actions on symplectic manifolds. In the rest of this paper we will
only consider  the semifree actions of $\C^*$ on Fano $3$-folds.

\section{ Types of fixed points for a semifree $\C^*$ action on a threefold}\label{sect3}
Assume $X$ is a Fano $3$-fold with an effective semi-free (free outside fixed point sets) $\C^*$-action.
Fix a circle subgroup $S^1$ of $\C^*,$ and choose a  $S^1$-invariant K\"ahler structure  on $X$
(it is always possible since $S^1$ is compact).
Since $X$ is simply connected we may further assume that $S^1$-action is Hamiltonian. Then it admits
a moment map $J:X\rw \R^1$ which is $S^1$-invariant and a perfect Bott-Morse function \cite{fr59},\cite{at82}.
The critical points of $J$ are exactly
the fixed points sets $X^{S^1}=X^{\C^*}.$ The connected components of  $X^{S^1}$ are K\"ahler
 hence algebraic submanifolds
of $X.$ Let $\{F_j\}$ be connected components of $X^{\C^*}$ and $X^{\pm}_j =\{x\in X|\lim_{t\rightarrow 0}t^{\pm}\cdot x\in F_j\}.$ Then we have
 $X=\cup X^+_j$ ({\rm resp.} $\cup X^-_j$), called   Bialynicki-Birula plus ({\rm resp.} minus) decomposition of $X$ \cite{bb}. The
 stable ({\rm resp.} unstable) submanifolds of $F_j$ defined by the gradient flows of $J$ coincides with $X^+_j$ ({\rm resp.} $X^-_j$).

For a fixed point $x$  in $ F_j,$ we consider the isotropic representation on  the tangent space $T_x X.$
It is a complex representation, and splits into direct sum of one dimensional irreducible representations. The tangent vectors of $T_x F_j$
are exactly the fixed vectors of the isotropic action (weighted zero vectors of the isotropic representation). The  vectors in the normal bundle $N_{F_j}X^{\pm}_j$
 are spanned by all vectors $v\in T_x X$ such that
for $t\cdot v=t^{\pm} v$ (weighted $\pm 1$ vectors of the isotropic representation).
 Let $\nu^0_j$ denote the complex dimension of $F_j$  and  $\nu^{\pm}_j$ denote the complex rank of $N_{F_j}X^{\pm}_j,$  then
     $$\nu^0_j +\nu^+_j +\nu^-_j =3.$$ We call $(\nu^-_j, \nu^+_j)$ the {\it weight type} of $F_j.$

If $\nu^+_j =3$ or $\nu^-_j =3$ then $x\in F_j$ is
 called an {\it elliptic} fixed point; if both $\nu^+_j$ and $\nu^- _j$
are nonzero then $x\in F_j$ is called a {\it hyperbolic} fixed point;
 if $\nu^0_j>0$ and one of  $\nu^+ _j, \nu^- _j$ is
zero then $x\in F_j$ is called a {\it parabolic} fixed point.
$2\nu^- _j$ is called {\it index} of $J$ as a Morse function at $F_j$.
Clearly the possible indices of the parabolic fixed points are $0,2,4.$ Hence the level sets $J^{-1}(v)$ is connected \cite{at82}
 and $J$ has a unique local minimum and a unique local maximum.

Let $F_{+}$ (resp. $F_-$) denote the connected component $F_j$  with $\nu^+_j =0$ (resp. $\nu^-_j =0$), and denote its weight type by $(n_-,0)$ with (resp. $(0,m_+)$ ).
 Then on $F_-,F_+$ the moment map $J$
take minimum and maximum respectively. Write $\C^*=S^1\times \R_{\geq 0},$
then the moment map is $S^1$-invariant and
 an increasing function along the $\R_{\geq 0}$-orbits \cite[the proof of Proposition 4.1]{yb}.  Hence the points in the
other components of $X^{\C^*}$ are all consisted of hyperbolic fixed points. In particular each connected component of $X^{\C^*}\setminus \{F_-\cup F_+\}$ is either an isolated point or a smooth algebraic curve.

Let $\{C_k|k=1,\cdots,a\}$ denote the set of fixed curves locating in $X^{\C^*}\setminus \{F_-\cup F_+\},$ they are called {\it interior fixed curves} since
on them the moment map take values between maximum and minimum.
 Then all points in the interior fixed curves $\{C_k\}$ have indices $2,$ hence their  weight type is $(\nu^-,\nu^+)=(1,1).$ The isolated fixed points locating in $X^{\C^*}\setminus \{F_-\cup F_+\}$
are called {\it interior fixed points}, and they are
   decomposed into two classes. One class consists of points with indices $2$, denoted by $\{x_i|i=1,\cdots b,\}.$ The other class consists of points with indices $4$, denoted by $\{y_j|j=1,\cdots, c\}$.  Hence $\{x_i\}$ have weight type
$(\nu^-,\nu^+)=(1,2)$ and $\{y_j\}$ have weight type $(\nu^-,\nu^+)=(2,1).$

To sum up, we have
\begin{proposition}\label{str} The fixed points set $X^{\C^*}$ has the following decomposition
$$X^{\C^*}=F_- \bigcup F_+ \cup \bigcup_{k=1}^a\{C_k\}\cup\bigcup_{k=1}^b\{x_i\}\cup\bigcup_{k=1}^c\{y_j\},$$
where on $F_-$ and $F_+$ the moment map take minimum and maximum respectively, $F_-$ has weight type $(0,m_+)$ with $m_+ =1,2$ or $3$  and $F_+$ has weight type $(n_-,0)$ with $n_- =1,2$ or $3;$
 the interior fixed curves $\{C_k\}$ has weight type $(1,1);$ the interior fixed points  $\{x_i\}$ and $\{y_i\}$ have weight type $(1,2)$ and weight type $(2,1)$ respectively.
\end{proposition}

\section{The Hodge numbers of the fixed submanifolds for a semifree $\C^*$ action on a Fano $3$-fold }\label{sect4}
Let $X$ be a Fano $3$-fold with a algebraic $\C^*$-action, as studied in Sect. \ref{sect3}. Write its $p$-th homology group as a sum of free subgroup and torsion (finite) subgroup:   $H_p(X,\Z)=\Z^{b_p}\oplus T_p,$ where $b_p$ is the $p$-th Betti number
 and $T_p$ is the torsion part of $H_p(X,\Z).$
From the universal coefficient theorem we have $H^{p}(X,\Z)=\Z^{b_p}\oplus T_{p-1}.$
Applying Poicar\'e duality we have $b_p=b_{6-p}$ and $T_p\cong T_{5-p}.$ Clearly $b_0=b_6=1,b_1=b_5=0$
and  $T_0=T_5=T_1=T_4=0,T_2\cong T_3.$
Therefore the integral cohomology of $X$ is decided as the following if we know the its $2$-homology group and $3$-th Betti number:
$$\begin{array}{rcl}
&&H^0(X,\Z)=H^6(X,Z)= \Z,\quad H^1(X,\Z)=H^5(X,Z)=0,\\
&&H^2(X,\Z)=\Z^{b_2},\quad H^3(X,\Z)=\Z^{b_3}\oplus T_2,\quad H^4(X,\Z)=\Z^{b_2}\oplus T_2.\\
\end{array}
$$
\begin{proposition}\label{torfree} Let $X$ be a Fano  $3$-fold equipped with algebraic $\C^*$-action such that the restricted
circle subgroup action is Hamiltonian, then $T_2=0. $ In particular, all  homology  and cohomology groups with integer coefficient of $X$ are torsion free.
\end{proposition}
\begin{proof}
In traditional Morse theory,
 $X$ is built up by attaching negative  normal bundles of critical submanifolds, which is equivalent to give $X$ a  Bialynicki-Birula  minus decomposition.
 In \cite{cs79} and \cite{fu79} the authors used  Bialynicki-Birula  plus decomposition to study the geometry of $\C^*$-manifolds. In this paper we use minus decomposition for notation convenient, though using either plus or minus decomposition cause no essential difference.
Let $\{F_j\}$ be the connected components of the fixed point
 of the $\C^*$-action with  weight type $(\nu^-_j, \nu^{+}_j).$ Carrell and Sommese \cite{cs79}
proved that there are isomorphisms of homology
$$\bigoplus_j H_{p-2\nu^-_j}(F_j,\Z)\longrightarrow H_p(X,\Z),$$
for any integer $p.$ Using the notions of Proposition \ref{str}, we have
$$H_2(X,\Z)\cong \bigoplus_{k=1}^a H_2(C_k,\Z)\oplus \Z^{b}$$
is torsion free. Therefore $T_2=0.$
\end{proof}

The following theorem give a diffeomorphic classification of some special $6$-dimensional manifolds by their numerical invariants.
\begin{theorem}\label{wal}(\cite[Theorem 5]{wall}).
 If two simply-connected $6$-dimensional symplectic manifolds
have no torsion homology, and if their second Stiefel-Whitney classes $w_2$ vanish and
they have the same cohomology ring and the same first Pontryagin class $p_1,$ then they are diffeomorphic.
\end{theorem}

Since $-K_X$ is a positive line bundle, by Kodaira's vanishing theorem,
             $$H^k(X,{\mathscr O}_X)=H^{0,k}\cong H^{k,0}=0, \quad k=1,2,3.$$
By Hodge decomposition theorem and Serre duality theorem, the Dolbeault cohomology groups and de Rham cohomology groups are respectively given by

$$\begin{array}{rcl}
&& H^{1,3}=H^{3,1}=H^{2,3}=H^{3,2}=0,\\
&&H^2(X,\C)=H^{1,1}, \quad H^3(X,\C)=H^{1,2}\oplus H^{2,1}, \quad  H^4(X,\C)=H^{2,2}\cong H^{1,1}.
\end{array}
$$
Therefore
$$b_2 =h^{1,1},\quad b_3 =2h^{1,2}.$$

\begin{proposition}\label{maxmin}
Suppose that a  connected component $S$ of $X^{\C^*}$ is a surface,
then it is a Fano surface. Moreover the moment map takes
maximum or minimum on $S.$ In particular, the number of complex
$2$-dimension components in $X^{\C^*}$ is not more than $2.$
\end{proposition}
\begin{proof}
The second claim is clear since the normal bundle of $S$ is a
complex $1$-dimensional representation of $\C^*,$ its weight is
either positive or negative. Without loss of generality we assume on
$S$ the moment map takes minimum and we can further assume the
minimum is $0.$

 By the equivariant Darboux theorem \cite{at82},
there is a local holomorphic coordinate $(z_1,z_2,z_3)$ around a point $x\in S$ such
that $S$ is defined by $z_3=0,$ the action is given by
$$t\cdot(z_1,z_2,z_3)=(z_1,z_2,t z_3),$$
and the moment map $J$ has the following local expression
$$J(x)=|z_3|^2 +o(r^2),$$
where $r^2=|z_1|^2+|z_2|^2+|z_3|^2.$ Let $\epsilon$ be a
sufficiently small positive number. Then $J^{-1}(\epsilon)$ is a
circle bundle over $S=P_{0}=J^{-1}(0),$ and $P_{<\epsilon}=\{x\in X|J(x)<\epsilon\}$ is a complex line bundle over $P_0$ whose fibres are exactly the orbits of the induced $\C^*$-action. If we contract the fibres of $P_{<\epsilon}\rightarrow P_0$ along the $\C^*$-orbits, the quotient space is exactly $P_0.$ Let $$P^{ss}_{\epsilon}=\{x\in P_{<\epsilon}|\overline{\C^* \cdot x}\cap J^{-1}(\epsilon)\not=\emptyset\}$$
denote the semistable set. Note $P^{ss}_{\epsilon}=P_{<\epsilon}$ since each $\C^*$-orbit intersects with the level set $J^{-1}(\epsilon)$
 if $\epsilon$ is sufficiently small. On the other hand  by the main theorem in
\cite{hl94} we know the Marsden-Weinstein quotient is isomorphic to the GIT quotient
$$J^{-1}(\epsilon)/S^1\cong P^{ss}_{\epsilon}\slash\slash \C^*=P_{<\epsilon}\slash\slash \C^*.$$
Hence $J^{-1}(\epsilon)/S^1$ is holomorphicly isomorphic to $P_0.$
By the main theorem in the Futaki's paper \cite{fut},
$J^{-1}(\epsilon)/S^1$ is  a Fano manifold if $X$ is Fano. Therefore $S=P_0$ is a del Pezzo surface.
\end{proof}
\begin{proposition} \label{hodgenumber} Let $X$ be a Fano $3$-fold with semifree algebraic $\C^*$-action, and notion $F_+,F_-,C_k, a,b,c$ defined as in the paragraph ahead of Proposition \ref{str}. Then

\begin{enumerate}
\item If maximum component $F_+$ (resp.  minimum component $F_-$) is a curve,
then it is a rational curve;
\item We could read Hodge numbers
of $X$ form the fixed point data:

$$\begin{array}{rcl}
h^{1,1}&=&h^{1-n_-,1-n_-}(F_+)+h^{1,1}(F_-)+a+b;\\
&=&h^{2-n_-,2-n_-}(F_+)+h^{2,2}(F_-)+a+c;\\
h^{1,2}&=&\sum_{k=1}^a g_k;
\end{array}$$
where $n_-$ denote the number of negative weights of isotropic
representation of $\C^*$ at the points in $F_+$ and $g_k$ is the genus of the fixed curve $C_k.$
\end{enumerate}
\end{proposition}

\begin{proof}
Carrell, Sommese \cite{cs79} and Fujiki \cite{fu79} proved that
there are isomorphisms of Dolbeault cohomology groups
$$\bigoplus_j H^{p-\nu^-_j,q-\nu^- _j}(F_j)\longrightarrow H^{p,q}(X),$$
for all $p,q.$  Hence
$$\begin{array}{rcl}
H^{1,1}(X)&\cong & H^{1-n_-,1-n_-}(F_+)\oplus H^{1,1}(F_-)\oplus
\bigoplus_{k=1}^a H^{0,0}(C_k)\oplus \Z^{b};\\
 H^{2,2}(X)&\cong &
H^{2-n_-,2-n_-}(F_+)\oplus H^{2,2}(F_-)\oplus \bigoplus_{k=1}^a
H^{1,1}(C_k)\oplus \Z^{c};\\
H^{1,2}(X)&\cong & H^{1-n_-,2-n_-}(F_+)\oplus H^{1,2}(F_-)\oplus
\bigoplus_{k=1}^a
H^{0,1}(C_k).\\
\end{array}$$
From the first two equations we have the first part of (B). By the
same reasoning, we have
$$H^{0,p}(X)=H^{0,p}(F_-)=0,\quad {\rm for}~p=1,2,3.$$
Hence if $F_-$ is a curve then $F_-=P^1.$ Note $F_+$ and $F_-$ have
symmetric positions (if we turn the manifold $X$ upside down along the direction where the moment map increases,  then $F_+$ and $F_-$ will exchange positions), so $F_+=P^1$ if it is a curve. By Proposition
\ref{maxmin}, $F_+,F_-$ are del Pezzo surfaces when they are
surfaces. Therefore we always have
$$H^{p,q}(F_+)=H^{p,q}(F_-)=0,\quad {\rm if}~p\not=q$$
whenever they are points, curves or surfaces. Hence $H^{1,2}(X)\cong
\oplus_{k=1}^a H^{0,1}(C_k),$ we get the second part of  (B).
\end{proof}

\section{Computations for the holomorphic Lefschetz fixed point formula and equivariant localization formula}\label{sect5}
By Atiyah and Bott's  Lefschetz fixed point formula \cite{ab}
together with Atiyah and Singer's index theorem \cite{as} we have
the following holomorphic Lefschetz fixed point formula:

\begin{theorem}\label{lef} Let $X$ be a compact complex manifold, $E$ a holomorphic vector bundle over $X,$
and $G$ a compact Lie group. Suppose that $g\in G$ is a holomorphic automorphism of pair $(X, E).$
 Let $\{F_i\}$
be the connected components of the fixed point
 of $g$ and suppose they are smooth submanifolds. The normal bundle $N_{F_i}$ of $F_i$
 under action of $g$ is split into direct sum of line bundle
$N_{F_i}=\oplus_{j} L_j,$ where $g$ acts on $L_j$ by multiplying
$t^{n_j}$ and denote the first Chern class of $L_j$ by $x_j$. Then
the Lefschetz number
\begin{eqnarray}\sum_i (-1)^i {\rm Tr}\Big(g|_{H^i(X, {\mathscr O}(E))}\Big)=\sum_{i}L(F_i, {\mathscr O}(E)),\label{lefsc}\end{eqnarray}
where the contribution of the component $F_i$ is computed as
\begin{eqnarray}L(F_i,{\mathscr O}(E))=\int_{F_i}\frac{ch(E|_{F_i})\Pi_j
\Big(\frac{1-e^{-x_j}/t^{n_j}}{1-1/t^{n_j}}\Big)^{-1} td(F_i)} {\det
(1-g|_{(N_{F_i})^*})},
\label{lefs}\end{eqnarray}
 where $ch(E|_{F_i})$ is the Chern character
of $E|_{F_i}$ and $td(F_i)$ is the total Todd class of $F_i.$

\end{theorem}

The following proposition is special case of Lemma 3.1 of \cite{tw00}, where it is proved in another way. Here we use Theorem \ref{lef} and our proof is very simple.

\begin{proposition} Let $X$ be a Fano-$n$-fold equipped with
a semi-free circle action with isolated fixed points. Let $N_k$
denote the number of fixed points of index $2k$ for $0\leq k\leq n.$ Then $N_k=\frac{n!}{k!(n-k)!}.$
\end{proposition}

\begin{proof}
If $x$ is a fixed point of index $2k,$ then $L(x,{\mathscr
O}_X)=\frac{1}{(1-t)^k (1-t^{-1})^{n-k}}.$ Note $H^i(X,{\mathscr
O}_X)=0$ for $i\geq 1$ hence $\sum_i (-1)^i {\rm Tr}(g|_{H^i(X, {\mathscr O}_X)})=1$ for any biholomorphic map $g.$
 By the holomorphic
Lefschetz fixed point formula (\ref{lefsc}) we have
\begin{eqnarray}\begin{array}{rcl}1=\sum_{k=0}^n\frac{N_k}{(1-t)^k (1-t^{-1})^{n-k}}=\frac{1}{(1-t)^n}\sum_{k=0}^n N_k(-t)^{n-k},\end{array}\end{eqnarray}
holds for any $t\in\C^*,$ hence $N_k=\frac{n!}{k!(n-k)!}.$
\end{proof}

By \cite[Theorem 1.1 and Theorem 1.2]{tw00}, $X$ and $(P^1)^n$ (the product of complex projective line) have the same cohomology rings and the same Chern classes. When $n=1,2$ it is easy to see they are biholomorphic isomorphic by classifications of curves and surfaces, but for $n>3 $ we don't know whether they are biholomorphic isomorphic. Using Wall's Theorem \ref{wal} and the classification of Fano 3-folds we know they are biholomorphic isomorphic when $n=3:$

\begin{corollary}\label{p3}  If $X$ be a Fano-$3$-fold equipped with
a semi-free circle action with isolated fixed points. Then
$X=(P^1)^3.$
\end{corollary}
\begin{proof}By  \cite[Theorem 1.1 and Theorem 1.2]{tw00}, $X$ and $(P^1)^3$ have the same cohomology ring and equivariant cohomology ring, and they both are homology torsion free by Proposition \ref{torfree}. In particular they have the same Chern classes $c_1,c_2,c_3.$ Hence they have the same Stiefel-Whitney classes $w_2\equiv c_1 {\rm mod 2}$ and Pontryagin class $p_1=c_1^2-2c_2.$ By Theorem \ref{wal}, $X$ is diffeomorphic to $(P^1)^3.$ Since $X$ is Fano, it is easy to see $X=(P^1)^3$ by \cite{mm1,mm2} Mori-Mukai's classification of Fano $3$-folds.
 \end{proof}
In the rest of this section we assume $X$ is a Fano $3$-fold with a
$\C^*$-action and write $X^{\C^*}=F_+\cup F_- \cup_{k=1}^a
C_k\cup\{x_i,y_j|1\leq i\leq b,1\leq j\leq c\}$ as in Section
\ref{sect3}. The contributions of isolated fixed points to the
Lefschetz fixed point formula are given by
$$\begin{array}{rcl}L(x_i,{\mathscr
O}_X)&=&\frac{1}{(1-t)(1-t^{-1})^2}=\frac{t^2}{(1-t)^3};\\
L(y_j,{\mathscr
O}_X)&=&\frac{1}{(1-t)^2(1-t^{-1})}=\frac{-t}{(1-t)^3}.\\
\end{array}$$
The normal bundle of the interior fixed curve $C_k$ splits as
$$N_{C_k}=L^+_k \oplus L^-_k ,$$
where $\C^*$ acts on $L^+_k$(resp. $L^-_k$) by positive (resp.
negative) weight. Let $c_1(L^{\pm}_k)$ denote the first Chern class
of $L^{\pm}_k.$  Since $Ch({\mathscr O}_X|_{C_k})=1$ and $Td(C_k)=1+\frac{c_1(C_k)}{2},$ by formula (\ref{lefs}),
\begin{eqnarray}\begin{array}{rcl}
L(C_k,{\mathscr O}_X)=\int_{C_k}\frac{1+\frac{c_1(C_k)}{2}}{(1-te^{-c_1(L^-_k )})(1-t^{-1}e^{-c_1(L^+_k )})}
\end{array}\end{eqnarray}
Using Taylor expansion of the exponential functions in the denominator of the integral (note it suffices to expand it to the linear terms since the integrated is over a curve):
\begin{eqnarray}\begin{array}{rcl}
L(C_k,{\mathscr O}_X)=\int_{C_k}\frac{1+\frac{c_1(C_k)}{2}}{(1-t(1-c_1(L^-_k)))(1-t^{-1}(1-c_1(L^+_k))}.\end{array}\label{eqq1}\end{eqnarray}
Rewrite two terms in  the denominator of the integral as

\begin{eqnarray}\left\{\begin{array}{rcl}
1-t(1-c_1(\!L^-_k\!)\!)&=&(1-t\!)(1+\frac{c_1(L^-_k\!) t}{1-t});\\
1-t^{-1}\!(\!1-c_1(L^+_k\!)\!)&=&(1-t^{-1}\!)(1+\frac{c_1(L^+_k\!) t^{-1}}{1-t^{-1}}\!).\\
\end{array}\right.\label{eqq2}\end{eqnarray}
 Taking (\ref{eqq2}) into (\ref{eqq1}) we get
\begin{eqnarray}\begin{array}{rcl}
L(C_k,{\mathscr O}_X)=\frac{1}{{(1-t)(1-t^{-1})}}\int_{C_k}\frac{1+\frac{c_1(C_k)}{2}}{(1+\frac{c_1(L^-_k) t}{1-t})(1+\frac{c_1(L^+_k) t^{-1}}{1-t^{-1}})}.\end{array}\label{eqq4}\end{eqnarray}
We have the following Taylor expansion of the reciprocal of the denominator of the integration of left hand side of (\ref{eqq4})   as a power series of variable of $\frac{1}{1-t},$
\begin{eqnarray}\begin{array}{rcl}
\frac{1}{(1+\frac{c_1(L^-_k) t}{1-t})(1+\frac{c_1(L^+_k) t^{-1}}{1-t^{-1}})}=1+\frac{c_1(L^+_k ) -c_1(L^-_k) t}{1-t}+\cdots,\end{array} \label{eqq5}
\end{eqnarray}
 note here it suffices to expand it to the linear term since the integration in (\ref{eqq4}) is on a complex curve. By (\ref{eqq4}) and (\ref{eqq5}) we have
\begin{eqnarray}\begin{array}{rcl}
L(C_k,{\mathscr O}_X)=\frac{1}{{(1-t)(1-t^{-1})}}\int_{C_k}\Big(\frac{c_1(C_k)}{2}+\frac{c_1(L^+_k ) -c_1(L^-_k) t}{1-t}\Big).\end{array}\label{eqq6}\end{eqnarray}
Note by Riemann-Roch Theorem for curve we have $\int_{C_k}c_1(C_k)=2\chi(C_k).$
For brevity  throughout of this paper we also use $c_1(L^{\pm}_k)$ to denote the first Chern number
$\int_{C_k}c_1(L^{\pm}_k),$ then we may rewrite (\ref{eqq6}) as
\begin{eqnarray}\begin{array}{rcl}
L(C_k,{\mathscr O}_X)
=\frac{-t(1-t)\chi(C_k)-(c_1(L^+_k ) -c_1(L^-_k) t)t}{(1-t)^3}.
\end{array}\end{eqnarray}
To sum up  we have

\begin{proposition}\label{lefinterior} For the interior fixed curve $C_k$ with normal bundle
$N_{C_k}=L^+_k\oplus L^-_k,$ the contribution $L(C_k,{\mathscr O}_X)$ in
the Lefschetz fixed point formula is computed as
$$\begin{array}{rcl}L(C_k,{\mathscr O}_X)=\frac{-t(1-t)\chi(C_k)-(\alpha^+_k -\alpha^-_k t)t}{(1-t)^3},
\end{array}$$
where $\alpha^{\pm}_k$ denote the Chern number $\int_{C_k}c_1 (\!L^{\pm}_k\! ).$
\end{proposition}

For the maximum and minimum fixed point components $F_+$ and $F_-,$
we use $N_{F_{\pm}}$ to denote their normal bundles in $X.$ We
still don't distinguish the Chern class and Chern number. For
example, if $F_+$ is a surface, $c^2_1(\!N_{F_+}\!)$ also denote the
integral $\int_{F_+} c^2_1\!(N_{F_+}\!).$

\begin{proposition}\label{leflocal} For the maximum and minimum fixed point components $F_+$ and $F_-,$
the contribution $L(F_{\pm},{\mathscr O}_X)$ is computed as in Table 1.
\end{proposition}
\begin{proof}
$F_{\pm}$ may be isolated points, rational curves or Fano surfaces.
Here we only give the computations  when $F_+$ is a surface, the
other cases are simpler and follow in the same way.

On $F_+$ the moment map $J$ takes maximum and the $\C^*$-action
on its normal bundle $N_{F_+}$ has a unique weight which  is
negative. Since $Ch({\mathscr O}_X|_{F_+})=1$ and $Td(F_+)=1+\frac{c_1(F_+)}{2}+\frac{c^2_1(F_+)+c_2(F_+)}{12},$ by formula (\ref{lefs}),
\begin{eqnarray}\begin{array}{rcl}
L(F_+,{\mathscr O}_X)=
\int_{F_+}\frac{1+\frac{c_1(F_+)}{2}+\frac{c^2_1(F_+)+c_2(F_+)}{12}}{1-te^{-c_1(N_{F_+})}}.\label{eqqq0}\end{array}\end{eqnarray}
Using Taylor expansion of the exponential function in the denominator of the integral
\begin{eqnarray}\begin{array}{rcl}
L(F_+,{\mathscr O}_X)=\int_{F_+}\frac{1+\frac{c_1(F_+)}{2}+\frac{c^2_1(F_+)+c_2(F_+)}{12}}
{1-t\Big[1-c_1(N_{F_+})+\frac{c_1^2(N_{F_+})}{2}\Big]}.\label{eqqq1}\end{array}\end{eqnarray}
Rewrite the denominator of the integral of the left hand side of (\ref{eqqq1}) as
\begin{eqnarray}\begin{array}{rcl}
{1-t\Big[1-c_1(N_{F_+})+\frac{c_1^2(N_{F_+})}{2}\Big]}=(1-t)\Big[1+\frac{t c_1(N_{F_+})}{1-t}-\frac{t c_1^2(N_{F_+})}{2(1-t)}\Big].
\label{eqqq2}\end{array}\end{eqnarray}
We have the following Taylor expansion of the reciprocal of $1+\frac{ tc_1(N_{F_+})}{1-t}-\frac{t c_1^2(N_{F_+})}{2(1-t)}$ as a power series of the variable $\frac{1}{1-t}:$
\begin{eqnarray}\begin{array}{rcl}
\frac{1}{1+\frac{t c_1(N_{F_+})}{1-t}-\frac{t c_1^2(N_{F_+})}{2(1-t)}}=1-\frac{tc_1(N_{F_+})}{1-t}+\frac{(1+t)tc_1^2(N_{F_+})}{2(1-t)^2}+\cdots, \label{eqqq3}
\end{array}\end{eqnarray}
note here it suffices to expand it to the quadric term since the integration in (\ref{eqqq1}) is on a complex surface. By (\ref{eqqq1}),  (\ref{eqqq2}) and (\ref{eqqq3}) we get
\begin{eqnarray}\begin{array}{rcl}
 L(F_+,{\mathscr O}_X)=\underset{F_+}{\int}\frac{(1+\frac{c_1(F_+)}{2}+\frac{c^2_1(F_+)+c_2(F_+)}{12})
(1-\frac{tc_1(N_{F_+})}{1-t}+\frac{(t+1)tc_1^2(N_{F_+})}{2(1-t)^2})}{1-t}.
\end{array}\end{eqnarray}
Note $F_+$ is a Fano surface, hence
$\chi({\mathscr O}_{F_+})=\int_{F_+}\frac{c^2_1(F_+)+c_2(F_+)}{12}=1.$
Therefore
\begin{eqnarray}\begin{array}{rcl}L(F_+,{\mathscr O}_X)=\frac{2(1-t)^2-t(1-t)c_1(F_+)c_1(N_{F_+})+t(1+t)
c_1^2(N_{F_+})}{2(1-t)^3}.
\end{array}\end{eqnarray}
\end{proof}

\begin{table}\label{fixeddata}
\begin{center}
\begin{tabular}{|c ||c |c|c |}
\hline Cases & point & line & surface \\ \hline\hline $L(F_-\!,\!\!{\mathscr O}_X\!)$ &
$\frac{-t^3}{(1-t)^3}$ & $\frac{(1-t\!)t^2\!+\!t^2
c_1\!(N_{F_-}\!)}{(1-t)^3}$&
$\frac{-2t\!(1-t\!)^2\!-\!t(1-t\!)c_1\!(F_-\!)c_1\!(N_{F_-}\!)\!-\!t(1+t\!)
c_1^2\!(N_{F_-}\!)}{2(1-t)^3} $
\\ \hline
$L(F_+\!,\!\!{\mathscr O}_X\!)$ & $\frac{1}{(1-t)^3}$ & $\frac{(1-t\!)\!-\!c_1\!(N_{F_+}\!)t}{(1-t)^3}$
&
$\frac{2(1-t\!)^2\!-\!t(1-t\!)c_1\!(F_+\!)c_1\!(N_{F_+}\!)\!+\!t(1+t\!)c_1^2\!(N_{F_+}\!)}{2(1-t)^3}
$
\\ \hline
\end{tabular}
\end{center}
\caption{Computations of $L(F_{\pm},{\mathscr O}_X)$}
\end{table}

Another important tool is  the Atiyah-Bott-Berline-Vergne
localization theorem \cite{au}:
\begin{theorem} \label{ablocal} Let $X$ be a compact symplectic manifold equipped with a Hamiltonian action of $S^1$
and denote $\{F_i\}$ the connected components of $X^{S^1}.$ Let
$\alpha\in H^*_{S^1}(X,\Z),$ Then
$$\begin{array}{rcl}\int_X \alpha =\sum_i\int_{F_i}\frac{\alpha|_{F_i}}{e^{S^1}(N_{F_i})},\end{array}$$
where $N_{F_i}$ is the normal bundle of $F_i$ and $e^{S^1}(N_{F_i})$
is its equivariant Euler class.
\end{theorem}

\begin{proposition}\label{rinterior} For the equivariant cohomology classes $ c^{S^1}_1 (X)$ and $ (c^{S^1}_1 (X))^3$ integrate at the interior fixed curve $C_k$  in
the Localization  formula is computed as
$$\begin{array}{rcl}\int_{C_k}\frac{1}{e^{S^1}(N_{C_k})}&=&\frac{\alpha^+_k -\alpha^-_k}{\lambda^3};\\
\int_{C_k}\frac{c_1^{S^1}(X)|_{C_k}}{e^{S^1}(N_{C_k})}&=&\frac{-c_1 (N_{C_k}) -c_1(C_k) }{\lambda^2};\\
\int_{C_k}\!\frac{(\!c_1^{S^1}\!(X)\!)^3|_{C_k}}{e^{S^1}(N_{C_k})}\!&=&0.\\
\end{array}$$
\end{proposition}

\begin{proof} We only give a proof of the second equation, the rest are proved in the same way. Since the normal bundle
 $N_{C_k}=L^+_k \oplus L^-_k$ and $\C^*$ acts on $L^{\pm}_k$ with weights $\pm 1,$ the equivariant Euler class $N_{C_k}$ is
 $e^{S^1}(N_{C_k})=(-\lambda +c_1( L^-_k)))(\lambda+c_1 (L^+_k)).$ Since $T_X|_{C_k}=T_{C_k} \oplus N_{C_k},$  and $\C^*$ acts trivially on the tangent bundle $T_{C_k}$
 and on the normal bundle $N_{C_k}$ with weight $(-1,1),$ the equivariant Chern class
 $c^{S^1}_1(X)|_{C_k}=c_1 (C_k) +c_1(N_k).$ Therefore
 $$\begin{array}{rcl}\int_{C_k}\frac{c_1^{S^1}(X)|_{C_k}}{e^{S^1}(N_{C_k})}&=&\int_{C_k}\frac{c_1 (C_k) +c_1(N_k)}{(-\lambda +c_1( L^-_k)))(\lambda+c_1 (L^+_k))}\\
                                                          &=&-\frac{1}{\lambda^2}\int_{C_k}[c_1 (C_k) +c_1(N_k)](1+\frac{c_1( L^-_k)}{\lambda})(1-\frac{c_1( L^+_k)}{\lambda})\\
                                      &=&\frac{-c_1 (N_{C_k}) -c_1(C_k) }{\lambda^2}.\\

\end{array}$$
\end{proof}

\begin{proposition}\label{rlocal} At the maximum and minimum fixed point components $F_+$ and $F_-,$ the integral of the equivariant cohomology classes $1, c^{S^1}_1 (X)$ and $ (c^{S^1}_1 (X))^3$
 in the Localization  formula are computed as in Table 2--4.
\end{proposition}
\begin{proof}
We only give the proof when $F_{\pm}$ are curves or surfaces in Table 4. The rest cases in Table 3 and Table 2 are simpler and proved in the same way.
If $F_{\pm}$ are rational curves, split $N_{F_{\pm}}=L^{\pm}_1\oplus L^{\pm}_2$ such that $\C^*$ act on  $N_{F_-}$ with weight $(1,1)$ and on $N_{F_+}$ with weight $(-1,-1)$ respectively. As the proof of the second equation in Proposition \ref{rinterior}, we have

$$\begin{array}{rcl}
&&\int_{F_{\pm}}\frac{(c_1^{S^1}(X))^3|_{F_{\pm}}}{e^{S^1}(N_{F_{\pm}})}\\
&=&\int_{F_{\pm}}\frac{(c_1 (F_{\pm})\mp 2\lambda +c_1(N_{F_{\pm}}))^3}{(\mp\lambda +c_1( L^{\pm}_1))(\mp\lambda+c_1 (L^{\pm}_2))}\\
&=&\frac{1}{\lambda^2}\int_{F_{\pm}}[\mp 8\lambda^3+12\lambda^2({c_1 (F_{\pm}) +c_1(N_{F_{\pm}})})]{(1 \pm \frac{c_1( L^{\pm}_1)}{\lambda})(1 \pm \frac{c_1 (L^{\pm}_2)}{{\lambda}})}\\
&=&\frac{1}{\lambda^2}\int_{F_{\pm}}[\mp 8\lambda^3+12\lambda^2({c_1 (F_{\pm}) +c_1(N_{F_{\pm}})})](1 \pm \frac{c_1(N_{ F_{\pm}})}{\lambda})\\
&=&\int_{F_{\pm}}[12{c_1 (F_{\pm})+4c_1(N_{F_{\pm}})}].\\
\end{array}$$
If $F_{\pm}$ are Fano surfaces, then $\C^*$ act on $N_{F_{\pm}}$ with weight $\pm 1$ hence
$$\begin{array}{rcl}
&&\int_{F_{\pm}}
\frac{c_1^{S^1}(X)|_{F_{\pm}}}{e^{S^1}(N_{F_{\pm}})}\\
&=&\int_{F_{\pm}}\frac{c_1 (F_{\pm})\mp \lambda +c_1(N_{F_{\pm}})}{(\mp\lambda+c_{1}(N_{F_{\pm}}))}\\
&=&\mp\frac{1}{\lambda} \int_{F_{\pm}}[c_1 (F_{\pm})\mp \lambda +c_1(N_{F_{\pm}})][1 \pm \frac{c_{1}(N_{F_{\pm}})}{\lambda}+\frac{c^2_{1}(N_{F_{\pm}})}{\lambda^2}]\\
&=&\mp\frac{1}{\lambda}\int_{F_{\pm}}[\pm\frac{(c_1 (F_{\pm}) +c_1(N_{F_{\pm}}))c_1(N_{F_{\pm}})}{\lambda}\mp \frac{c^2_{1}(N_{F_{\pm}})}{\lambda}]\\
&=&-\frac{c_1 (F_{\pm}) c_1(N_{F_{\pm}})}{\lambda^2}.\\

&&\\
&&\int_{F_{\pm}}
\frac{(c_1^{S^1}(X)|_{F_{\pm}})^3}{e^{S^1}(N_{F_{\pm}})}\\
&=&\int_{F_{\pm}}\frac{(c_1 (F_{\pm})\mp \lambda +c_1(N_{F_{\pm}}))^3}{(\mp\lambda+c_{1}(N_{F_{\pm}}))}\\
&=&\mp\frac{1}{\lambda} \int_{F_{\pm}}(c_1 (F_{\pm})\mp \lambda +c_1(N_{F_{\pm}}))^3[1 \pm \frac{c_{1}(N_{F_{\pm}})}{\lambda}+\frac{c^2_{1}(N_{F_{\pm}})}{\lambda^2}]\\
&=&\mp\frac{1}{\lambda} \int_{F_{\pm}}[\mp \lambda^3+3\lambda^2 (c_1 (F_{\pm}) +c_1(N_{F_{\pm}}))\mp 3\lambda (c_1 (F_{\pm}) +c_1(N_{F_{\pm}}))^2]\\
&&\cdot [1 \pm \frac{c_{1}(N_{F_{\pm}})}{\lambda}+\frac{c^2_{1}(N_{F_{\pm}})}{\lambda^2}]\\
&=& \int_{F_{\pm}} [c^2_1(N_{F_{\pm}})-3c_1(N_{F_{\pm}})(c_1 (F_{\pm}) +c_1(N_{F_{\pm}}))+3(c_1 (F_{\pm}) +c_1(N_{F_{\pm}}))^2]\\
&=&{3c^2_1({F_{\pm}})+c^2_1(N_{F_{\pm}})}+3c_1({F_{\pm}})c_1(N_{F_{\pm}}).
\end{array}$$

\end{proof}

\begin{table}
\begin{center}
\begin{tabular}{|c ||c |c|c |}
\hline Cases & point & line & surface \\ \hline\hline $\int_{F_-}\!\!\frac{1}{e^{S^1}(N_{F_-})}$ &
$\frac{1}{\lambda^3}$ & $-\frac{c_1 (N_{F_-})}{\lambda^3}$&
$\frac{c^2_1 (N_{F_-})}{\lambda^3} $
\\ \hline
$\int_{F_+} \!\!\frac{1}{e^{S^1}(N_{F_+})}$ & $-\frac{1}{\lambda^3}$ & $\frac{c_1 (N_{F_+} )}{\lambda^3}$
&
$-\frac{c^2_1 (N_{F_+})}{\lambda^3}
$
\\ \hline
\end{tabular}
\end{center}
\caption{Computations of $\int_{F_{\pm}}\!\!\frac{1}{e^{S^1}\!(\!N_{F_{\pm}}\!)}$}
\label{eqvs}
\end{table}

\begin{table}
\begin{center}
\begin{tabular}{|c ||c |c|c |}
\hline Cases & point & line & surface \\ \hline\hline $\int_{F_-}\frac{ c^{S^1}_1 (X)|_{F_-}}{e^{S^1}(N_{F_-})}$ &
$\frac{3}{\lambda^2}$ & $\frac{c_1 (F_-)-c_1(N_{F_-})}{\lambda^2}$&
$\frac{-c_1 (F_{-}) c_1(N_{F_{-}})}{\lambda^2}$
\\ \hline
$\int_{F_+} \frac{c^{S^1}_1 (X)|_{F_+ }}{e^{S^1}(N_{F_+})}$ & $\frac{3}{\lambda^2}$ & $\frac{c_1 (F_+)-c_1(N_{F_+})}{\lambda^2}$
&
$\frac{-c_1 (F_{+}) c_1(N_{F_{+}})}{\lambda^2}
$
\\ \hline
\end{tabular}
\end{center}
\caption{Computations of $\int_{F_{\pm}}\frac{ c^{S^1}_1 (X)|_{F_{\pm}}}{e^{S^1}(N_{F_{\pm}})}$}
\end{table}

\begin{table}
\begin{center}
\begin{tabular}{|c ||c |c|c |}
\hline Cases & pt & line & surface\!\! \\ \hline\hline $\int_{F_-}\!\!\frac{ (c^{S^1}_1 (X)\!)^3\!|_{F_-}}{e^{S^1}(N_{F_-}\!)}$ &
$27$ & ${12c_1\! (F_-\!)\!+\!4c_1\!(N_{F_-}\!)}$&
$3c^2_1\!({F_{-}}\!)\!+\!c^2_1\!(N_{F_{-}}\!)\!+\!3c_1\!({F_{-}}\!)c_1\!(N_{F_{-}}\!)$
\\ \hline
$\int_{F_+}\!\! \frac{(c^{S^1}_1 (X)\!)^3\!|_{F_+ }}{e^{S^1}(N_{F_+}\!)}$ & $27$ & ${12c_1\! (F_+\!)\!+\!4c_1\!(N_{F_+}\!)}$
&
$3c^2_1\!({F_{+}}\!)\!+\!c^2_1\!(N_{F_{+}}\!)\!+\!3c_1\!({F_{+}}\!)c_1\!(N_{F_{+}}\!)
$
\\ \hline
\end{tabular}
\end{center}
\caption{Computations of $\int_{F_{\pm}}\frac{ (c^{S^1}_1 (X))^3|_{F_{\pm}}}{e^{S^1}(N_{F_{\pm}})}$}
\end{table}
As in the last section,  we may consider $X$ as a K\"ahler manifold with an algebraic
$\C^*$-action such that the circle subgroup $S^1$-action is
Hamiltonian with moment map $J.$ Denote the level set $J^{-1}(r)$ by $P_r.$ If $r$
is a regular value of the moment map $J,$ by \cite{hl94} the $\C^*$ act freely
on $X(r):=\{x\in X|\overline{\C^*\cdot x} \cap P_r\not=\emptyset \}$
and the GIT quotient $X(r)//\C^*$ is isomorphic to the
Marsden-Weinstein quotient $X_r:=P_r/S^1,$ and it is a K\"ahler
manifold. Hence $P_r$ is a circle bundle over the K\"ahler manifold
$X_r,$ we denote its Euler class by $e(P_r)\in H^2(X_r,\Z).$
Assume that $\epsilon\geq 0$ is small enough such that $c$ is the only
critical value in $[c-\epsilon,c+\epsilon].$  Using symplectic
cutting \cite{yb} we get a $6$-dimensional K\"ahler manifold
$X^{c+\epsilon}_{c-\epsilon}$ with semifree $\C^*$-action, which is
a $\C^*$-equivariant K\"ahler compactification of
$J^{-1}((c-\epsilon,c+\epsilon)).$ The  fixed points of the
$\C^*$-action on the cutting space $X^{c+\epsilon}_{c-\epsilon}$
consisting of three parts:  $X_{c-\epsilon},$ the  connected component where the
moment map takes minimum; $ X_{c+\epsilon},$  the connected component  where the
moment map takes maximum; and those fixed point in the critical level set
$J^{-1}(c).$

\begin{proposition}\label{localization} Let $\epsilon\geq 0$ be small number such that $c$ is the only
critical value in $[c-\epsilon,c+\epsilon].$
 Assume there are $k$ fixed points in $J^{-1}(c)$
whose wights of isotropic representation are $(-1,-1,1)$ and $l$ fixed points in $J^{-1}(c)$
whose wights of isotropic representation are    $(-1,1,1),$ and there are $m$ fixed curves $C_1,\cdots,C_m$  in $J^{-1}(c)$ with normal bundles
$N_{C_k}=L^+_i\oplus L^-_i$ for $i=1,\cdots,m.$
Then
\begin{eqnarray}\begin{array}{rcl} e^2(P_{c+\epsilon})- e^2(P_{c-\epsilon})=k-l+\sum_{k=1}^m(c_1(L^+_k)-c_1(L^-)).\end{array}\label{eqic}\end{eqnarray}
If $c$ is a minimum or a maximum of $J$ then $J^{-1}(c)$ is connected, in this case we have
$e^2(P_{c\pm \epsilon})=1$ when  $J^{-1}(c)$ is point, and $e^2(P_{c\pm \epsilon})=-c_1(N_C)$ when $J^{-1}(c)$ is a curve $C.$
\end{proposition}
\begin{proof}
On the cutting space, integrate $1\in
H^*_{S^1}(X^{c+\epsilon}_{c-\epsilon},\Z)$ via using Theorem
\ref{ablocal}, then
$$\begin{array}{rcl}&&\int_{X_{c-\epsilon}}\frac{1}{\lambda+e(P_{c-\epsilon})}+
\int_{X_{c+\epsilon}}\frac{1}{-\lambda-e(P_{c+\epsilon})}\\
&+&\frac{k}{(-\lambda)^2\lambda}+\frac{l}{(-\lambda)\lambda^2}+\sum_{i=1}^m\int_C\frac{1}{(\lambda+c_1(L^+_i))
(-\lambda+c_1(L^-_i))}=0.\end{array}$$
expanding it we get (\ref{eqic}).  We consider the cutting space $X^{c+\epsilon}_{c}$ if $c$ is a minimum and the cutting space $X^c_{c-\epsilon}$ if $c$ is a maximum, integrate $1$ on these spaces we get the second part of Proposition \ref{localization}.
\end{proof}

\section{Classification when $X^{\C^*}$ has exactly two connected components and examples for general $\C^*$-actions }\label{sect6}

\begin{theorem} \label{two}Suppose that $X^{\C^*}$ has exactly two connected component $F_+$ and $F_-.$
\begin{enumerate}
\item If both $F_+$ and $F_-$ are isolated fixed points, such a Fano $3$-fold $X$ does not
exist;

\item  If both $F_+$ and $F_-$ are fixed surfaces, then $F_+$
is biholomorphic to $F_-$ and $X$ is a $P^1$ bundle over $F_+;$

\item  If one of $F_+$ and $F_-$ is a fixed surface and the other is an isolated fixed point,
then $X=P^3;$

\item  If one of $F_+$ and $F_-$ is a fixed curve and the other is an isolated fixed
point, such $X$ does not exist.
\end{enumerate}
\end{theorem}

\begin{proof}
 (A). The moment map $J$ is a nondegenerate Morse function, it is well know in this case that $X$ is diffeomorphic to
$S^6,$ but $S^6$ is not even  a symplectic manifold. (B). Since  the the Marsden-Weinstein reduced space $X_r=J^{-1}(r)$ will not change  diffeomorphism type if $J$ does not cross its critical value, we denote the Marsden-Weinstein reduced space by $X_{red}.$ From the proof of
Proposition \ref{maxmin}, both $F_+$ and $F_-$ are isomorphic to the Marsden-Weinstein reduced space
$X_{red}$ and the $\C^*$-action induces a holomorphic $P^1$ bundle
structure of $X.$ (C). Assume that $J(X)=[0,h],$ and without loss of
generality we assume $J^{-1}(0)$ is a point and $J^{-1}(h)$ is
a surface. For a sufficiently small positive number $\epsilon<h,$ we
know $X_{\epsilon}=J^{-1}(\epsilon)/S^1$ is $P^2.$ Hence
$X_{h-\epsilon}\cong X_{red}\cong P^2$ and $X$ is obtained by attaching a cell of index zero to $P^2,$ therefore
$X=P^2\cup_{J^{-1}(h-\epsilon)}D^6\cong P^3.$ (D). Near one
vertex of the interval $J(X),$ the Marsden-Weinstein reduced space $X_{red}$ is isomorphic to $P^2$ and near the other vertex of the interval
$J(X),$  the Marsden-Weinstein reduced space $X_{red}$ is a holomorphic $P^1$ bundle over $P^1.$
 It
is impossible since $P^2$ is not diffeomorphic to the $P^1$ bundle over
$P^1.$
\end{proof}

\begin{example}\label{exp1}{\rm Clearly $P^3$ admits  semifree $\C^*$-actions whose fixed points
$X^{\C^*}=P^2\cup{pt}$ or $P^1\cup P^1.$ For example for the action
$$t\cdot[z_0:z_1:z_2:z_3]=[tz_0:tz_1:z_2:z_3],$$
has a moment map
$$J([z_0:z_1:z_2:z_3])=\frac{|z_0|^2+|z_1|^2}{|z_0|^2+|z_1|^2+|z_2|^2+|z_3|^2}.$$
Clearly $J(P^3)=[0,1]$ and $X^{\C^*}=P^1\cup P^1.$}
\end{example}

\begin{example}\label{exp11}{\rm For the semifree $\C^*$-action on $P^3$ with two fixed lines
 $P^1\cup P^1.$ If we blow up one of the fixed line, then we get a Fano threefold  with semifree $\C^*$ action whose fixed point set still has two connected components, one is a line the other is a ruled surface.}
\end{example}

\begin{example}\label{exp3}{\rm
Let $X=Proj({\mathscr O}_{P^2}\oplus {\mathscr O}_{P^2}(k)),$
a $P^1$ bundle over $P^2$ with Euler number $k;$ or $X=Proj(T_{P^2}),$ the projective tangential bundle  of $P^2.$
 Let $\C^*$ act
  on the $P^2$ by $t\cdot
[z^0,z^1,z^2]=[tz_0,z_1,z_2],$ then the action is semifree, it induces natural action a the algebraic vector bundles $E$ and so does on the projective bundle $Proj(E).$ The fixed points
points have two connected components, one of them is $P^1$ and the other is a $P^1$ bundle over $P^1$. If $E=T_{P^2}$ or ${\mathscr O}_{P^2}\oplus {\mathscr O}_{P^2}(k)$ with $k\not=0$ then the induce actions on the projective bundle $Proj(E)$ is not semifree. When $E$ is a trivial bundle clearly there is an induced semifree action on  $Proj(E).$
}\end{example}

\begin{example}\label{exppaper1}{\rm
Let $\C^*$ acts  $ P^3$ by $$t\cdot([x_1,x_2,x_3,x_4])=([tx_1,x_2,x_3,x_4]).$$
The blow-up space $X$ of $P^3$ along the fixed line $\{[x_1,x_2,x_3,x_4]\in P^3|x_1=x_2=0\}$ is defined by
 $$\{([x_1,x_2,x_3,x_4],[y_1,y_2])\in P^3\times P^1|x_1y_2=x_2y_1\}.$$
 There is a lifted semifree action on $X$
  $$t\cdot([x_1,x_2,x_3,x_4],[y_1,y_2])=([x_1,tx_2,tx_3,tx_4],[y_1,ty_2])$$ with moment map
$$J([x_1,x_2,x_3,x_4],[y_1,y_2])=\frac{|x_2|^2+|x_3|^2+|x_4|^2}{|x_1|^2+|x_2|^2+|x_3|^2+|x_4|^2}+\frac{|y_2|^2}{|y_1|^2+|y_2|^2}.$$
The fixed point set $X^{\C^*}$ has three connected components: the maximum component $\{([0,x_2,x_3,x_4][0,1])\}\cong P^2,$ an interior fixed rational curve $\{([0,0,x_3,x_4][1,0])\}\cong P^1,$ the minimum component
is an isolated fixed point $\{([1,0,0,0][1,0])\}.$

}\end{example}

\begin{example}\label{exp5}{\rm
Let $\C^*$ acts  $P^2\times P^2$ by $$t\cdot([x_1,x_2,x_3],[y_1,y_2,y_3])=([x_1,x_2,t^2x_3],[ty_1,ty_2,y_3])$$
Let $X$ be the divisor of bidegree $(1,2)$ of $P^2\times P^2$ defined by $x_1y_1^2+x_2y_2^2 +x_3y_3^2=0.$  The
fixed points are two rational curves. But this action is  not semifree.
 }\end{example}

\begin{example}\label{exp4}{\rm
$SL_2(\C)$ acts on the polynomial ring $R = k[x,y]$ by $ x \rw ax+cy, y \rw
bx+dy$ for $s=\begin{bmatrix}a& b\\ c&d\end{bmatrix}\in SL_2(\C)$. Let $R_n\subset R$   denote the vector space of the
homogeneous polynomials of degree $n.$ Then $R_n$ is the irreducible $SL_2(\C)$-module
of degree $n+1.$ Let $f(x,y)=\sum_{j=0}^n C_n^j a_jx^{n-j}y^j $ be a non-zero homogeneous
polynomial. We take $(a_0,a_1, ,\cdots,a_n )$ as homogeneous coordinates on
$Proj(R_n )$ on which there is an induced $SL_2(\C)$ action. Take $f_6=xy(x^4-y^4)$ and $f_{12}=xy(x^{10}+11x^5y^5 +y^{10}).$
The {\it Mukai-Umemura Fano $3$-folds} \cite{mu} are defined by $X_5 =\overline{SL_2(\C)\cdot f_6}$ and $X_{12} =\overline{SL_2(\C)\cdot f_{12}},$ the closers of the orbits passing through $f_6$ and
$f_{12}$ respectively. Then there is naturally induced $\C^*$ actions on $X_5$ and $X_{22}.$ But these actions are not semifree since
they are the restrictions of the action $t\cdot [a_0:a_1:\cdots:a_n]=[t^{-n}a_0:t^{-n+2}a_1:\cdots:t^n a_n]$  on $Proj(R_n)$ with $n=5$ and $n=12$ respectively.
 }\end{example}

\begin{proposition}\label{blowu} Let $Y$ be a $3$-fold with a semifree  $\C^*$-action and $C\subset Y$ a smooth $C^*$-invariant curve.
Let $X$ be the blow-up of $Y$ along $C$ with exceptional divisor $E.$ Then there is a semifree action of $\C^*$ on $X$ such that $E$ is $\C^*$-invariant.
If $C$ is point-wise fixed then $E$ is point-wise fixed or the $\C^*$-action on $E$ has exactly two fixed curves isomorphic to $C.$
If $C$ is not point-wise fixed then  the  $\C^*$-action on $E$ has only isolated fixed points or fixed  curves.
\end{proposition}
\begin{proof}
Let $N_{C}$ be the normal bundle in $Y$ then $E$ is isomorphic to the projective bundle $Proj(N_C).$ If $C$ is point-wise fixed then the induced action of
$\C^*$ on the tangent bundle of $C$ must have weight $0.$ Since the action on $X$ is semifree, the induced action on the normal bundle $N_C$ must have nonzero weight.
 Hence it acts on the normal bundle $N_C$ with weights $(\pm 1,\pm 1)$ or $(0,\pm 1).$ In the former case  $E$ is  point-wise fixed, in the later case  the $\C^*$-action on $E$ has exactly two fixed curves, both of them are isomorphic to $C.$
Similarly if  $C$ is not point-wise fixed and the $\C^*$ acts on the fibre of $N_C$ with two different weights, then  the  $\C^*$-action on $E$ has only isolated fixed points; if  $C$ is not point-wise fixed and the $\C^*$ acts on the fibre of $N_C$ with the same weights, then  the  $\C^*$-action on $E$ has only fixed curves.
\end{proof}

\begin{remark}\label{blowmo}{\rm Let $Y$ be a $3$-fold with a semifree  $\C^*$-action and
 $X$ is the blow-up of $Y$ along a fixed point or an invariant curve $C$ with exceptional divisor $E.$
 If $E$ is not point-wise fixed then there are isolated fixed point or fixed curve on $E$  which are interior isolated fixed point or interior fixed curve of $X.$

 We take the curve blow up as an example.
 Suppose the curve $C$ is a fixed curve of the $\C^*$-action on $Y.$ Take holomorphic local coordinate $(y_1,y_2,y_3)$ of $Y$ such that $C$ is locally given by $y_1=y_2=0.$ Let $U$ be a $C^*$-invariant tubular neighborhood of $C$ in $Y$
Then there is a neighborhood $V$ of $C$ in $X$  which is equivariant isomorphic to $V=\{(y_1,y_2,y_3;[z_1,z_2])\in U\times P^1|z_1y_2=z_2y_1\}.$
 Since $E$ is not point-wise fixed,   the $\C^*$ action is given by
  $$\C^*\times (U\times P^1)\rw U\times P^1,t\cdot (y,[z_1,z_2])=(t\cdot y,[t^{\pm 1}z_1,z_2]).$$
  Suppose $J$ is the moment map for $S^1$ action on $Y,$ then near $C$ the moment map of  $S^1$ action on $V$ is given
  by $$\tilde{J}(y,[z_1,z_2])=J(y)\pm \frac{ |z_1|^2}{|z_1|^2+|z_2|^2}.$$
 Put $J(C)=a.$  Then on the exceptional divisor $E=Proj(N_C),$  the moment map $\tilde{J}$ take values $a\pm 1$ at the curve $\{z_2=0\}\cap E.$
  If $a$ is minimum (resp. maximum) of $J$ then $\tilde{J}$ take values $a +1$(resp. $a-1$) at the curve $\{z_2=0\}\cap E;$ if $C$ is an interior fixed curve of $Y$
  then $\{z_1=0\}\cap E$ or $\{z_2=0\}\cap E$ is  an interior fixed curve of $X.$
  Hence there is always a fixed curve of $E$ which is
    an interior fixed curve of $X.$

 If the curve $C$ is invariant but not point-fixed and  $E$ is not point-wise fixed, we could prove in the same way that there is at least one interior isolated fixed point of $X$ lying inside $E.$

}\end{remark}

\begin{proposition}\label{blowup} Suppose that algebraic manifold $X$ admits an algebraic $\C^*$-action, and has a blowing
down map $f:X\rightarrow Y$ along an exceptional divisor $E,$ then
$E$ is $\C^*$-invariant. Therefore, if the blowing up center
$f(E)$ is a curve of genus $g\geq 1,$ then there is a descent action
of $\C^*$ on $Y$ such that $f(E)$ is fixed.
\end{proposition}
\begin{proof}
Otherwise there exists $t\in \C^*$ such that $E':=t\cdot E\not=E.$
Take a rational curve $C$ in $E,$ then  the intersection number
$E\cdot C=\deg ({\mathscr O}_E(-1)|_C)=-1.$ But clearly $E\cdot C
=E'\cdot C=0,$ a contradiction. If $f(E)$ is a curve with genus
$g\geq 1,$   then the induced $\C^*$-action on $f(E)$ is trivial.
\end{proof}

\begin{theorem} Suppose that $X^{\C^*}$ has exactly two component $F_+$ and $F_-.$
\begin{enumerate}
\item If  $F_+=F_-=P^1$ , then  $X$
is  $P^3;$

\item If $F_+$ is a surface and $F_-=P^1,$ then $X=P^1\times P^2.$
\end{enumerate}

\end{theorem}

\begin{proof} (A). From (B) of Proposition \ref{hodgenumber} we know $h^{1,1}=1$ and $h^{1,2}=0.$
Using  the Lefschetz fixed point formula in Theorem \ref{lef} and the data in Table 1, we have
$$\begin{array}{rcl}1=\frac{(1-t)t^2+t^2c_1(N_{F_-})}{(1-t)^3}+\frac{(1-t)-c_1(N_{F_+})t}{(1-t)^3},\end{array}$$
hence we have $c_1(N_{F_+})=c_1(N_{F_-})=2.$ Using localization Theorem \ref{ablocal} and data computed in Table 4, we have
$$\begin{array}{rcl}-K^3_X=\int_X (c^{S^1}_1(X))^3={12c_1\! (F_+\!)\!+\!4c_1\!(N_{F_+}\!)}+{12c_1\! (F_-\!)\!+\!4c_1\!(N_{F_-}\!)}=64.\end{array}$$
By the classification for Fano threefolds with Picard number $\rho=h^{1,1}=1$ and $K^3_X=64$ in \cite[Appendix, Table \S 12.2]{ip}, $P^3$ is the only possibility.

  (B).
Using  the Lefschetz fixed point formula in Theorem \ref{lef} and the data in Table 1, we have
$$\begin{array}{rcl}1=\frac{(1-t)t^2+t^2c_1(N_{F_-})}{(1-t)^3}+\frac{2\!(1-t\!)^2\!-\!t(1-t\!)c_1\!(F_+\!)c_1\!(N_{F_+}\!)\!+\!t(1+t\!)
c_1^2\!(N_{F_+}\!)}{2(1-t)^3},\end{array}$$
we get
\begin{eqnarray}\left\{\begin{array}{rcl}
&&c_1(F_+)c_1\!(N_{F_+}\!)-c^2_1(N_{F_+})=2,\\
&&2c_1(N_{F_-})+c_1(F_+)c_1\!(N_{F_+}\!)+c^2_1(N_{F_+})=2.\\
\end{array}\right.\label{nff1}\end{eqnarray}
Hence we have $c_1(F_+)c_1\!(N_{F_+}\!)=c^2_1(N_{F_+})+2$ and $c_1\!(N_{F_-}\!)=-c^2_1(N_{F_+}).$
Integrate  $(c^{S^1}_1(X))^3$ on $X$ via using localization Theorem \ref{ablocal} and data computed in Table 3, we get
$$-K^3_X=12c_1(F_-)+4c_1 (N_{F_-}) +3c^2_1({F_{+}})+c^2_1(N_{F_{+}})+3c_1({F_{+}})c_1(N_{F_{+}}),$$
note $F_-=P^1$ and $F_+$ is a ruled surface hence $c_1(F_-)=2$ and $c^2_1(F_+)=8,$ thus
$$-K^3_X=48+4c_1 (N_{F_-})+c^2_1(N_{F_{+}})+3c_1({F_{+}})c_1(N_{F_{+}})=54.$$
By the classification for Fano threefolds with Picard number $\rho=h^{1,1}=2$ and $K^3_X=54$ in \cite[Appendix, Table \S 12.3]{ip},
 there are two such Fano $3$-folds, $P^1\times P^2$ and the blow-up of $P^3$ along a line. Note we have realized such actions as in Example
 \ref{exp11} and \ref{exp3}.
\end{proof}

\begin{table}
\begin{center}
\begin{tabular}{|c ||c |c|c | c|}
\hline $X^{\C^*}$& $h^{1,1}$ & $h^{1,2}$ & X
\\ \hline\hline ${\rm two~ points}$ & 0& 0& {\rm not~ exist} \\ \hline
${\rm a~~ point~~and~~ a ~~line}$& 1 & 0& {\rm not~ exist} \\ \hline
${\rm a~~ point~~and~~ a ~~Fano ~~ surface}$& 1 & 0& $P^3$ \\ \hline
${\rm two ~~lines}$& 1&0&$P^3$
 \\ \hline
${\rm a~~ line~~ a ~~Fano ~~ surface}$ &2&0&
 $P^1\times P^2,{\rm blow~~up~~of} P^3 {\rm along~~a~~line}$ \\
 \hline
${\rm two~~~Fano~~~surfaces}$&1+d &0& {\rm $P_{S_d} (E) $} \\
\hline
\end{tabular}
\end{center}
{(\tiny {\bf Remark:} $E$ is rank 2 vector bundle on the Fano surface $S_d$ of Picard number $d$)}
\caption{Case a=b=c=0}
\end{table}

\begin{remark}{\rm By the classifications of Fano threefolds \cite[Table \S 12.2]{ip}, there are only four of them with Hodge number
$h^{1,1}=1$ amd $h^{1,2}=0,$ they are $P^3,$ quadric
$\widetilde{Gr}_2(\R^5),$  Mukai-Umemura Fano manifolds $X_5$ and $X_{22}.$
There are many examples \cite[Table \S 12.3]{ip} of Fano
$3$-folds $X$ with $h^{1,1}(X)=2,h^{1,2}(X)=0,$ most of them are obtained
from the blowing up of those Fano $3$-folds $Y$ with
$h^{1,1}(Y)=1,h^{1,2}(Y)=0$ along  curves.   By Proposition \ref{blowup} and \ref{blowu}
 they are excluded. The rest cases \cite[Table \S 12.3]{ip}: a divisor on $P^2\times P^2$ of bidegree  $(1,2), $ the projective tangential bundle $Proj(T_{P^2})$ of $P^2,$ or $Proj({\mathscr
O}_{P^2}\oplus {\mathscr O}_{P^2}(k))$ as  a $P^1$ bundle over $P^2$ with
Euler number $k=0,1,2.$ }
\end{remark}

\section{Classification where $X^{\C^*}$ has no interior isolated fixed points but at least one interior fixed curve}\label{sect7}
 Let $X$ be a Fano
$3$-fold with an algebraic $\C^*$-action such that the circle
subgroup $S^1$-action is Hamiltonian. We denote the level set
$J^{-1}(r)$ by $P_r.$ If $r$ is a regular value of $J$ then the Marsden-Weinstein reduced space$X_r:=P^r/S^1$ is a smooth surface and $P_r\rw X_r$ is a principle circle bundle.
 In \cite{fut} Futaki proved that
$X_r$ is a Fano surface when $r$ is a regular value.
 We denote the
Euler class of the circle bundle  by $e(P_r)\in
H^2(X_r,\Z).$

 Let $c$ be a critical value of $J,$ and suppose  $C_1,\cdots,C_k
\subset P_{c}$ all algebraic curve locating inside $P_c$ that fixed by the $\C^*$-action and there is no other fixed points.
 Assume that $\epsilon$ is small enough such that
$c$ is the only critical value in $[c-\epsilon,c+\epsilon].$ By
\cite{gs89}, $X_{c-\epsilon}$ is diffeomorphic to $X_{c+\epsilon}.$ Therefore if the fixed point set $X^{\C^*}$ has no interior isolated fixed points,
then the  Marsden-Weinstein reduced spaces have the same diffeomorphism type,
let's denote it by $X_{red}.$ Moreover $C_1,\cdots,C_k$ are embedding
K\"ahler submanifold of $X_{red}.$ We denote the dual class of them
by $[C_1],\cdots,[C_k],$ then the Euler classes of the circle
bundles vary by
\begin{eqnarray}\label{euler}
e(P_{c+\epsilon})=e(P_{c-\epsilon})+[C_1]+\cdots+[C_k].\end{eqnarray}
Let $N_{C_i |X}$ (resp. $N_{C_i |X_{red}}$) denote the normal bundle
of $C_i$ in $X$ (resp. $X_{red}$). Then $N_{C_i |X}$ has a
decomposition $N_{C_i |X}=L^+_i \oplus L^-_i,$ and on $L^+_i$ (resp.
$L^-_i$) the $\C^*$-action by positive (resp. negative) weight.
$N_{C_i |X_{red}}$ is isomorphic to $L^+_i \otimes L^-_i.$ Let
$\alpha_i^{\pm}=c_1(L^{\pm}_i)$ denote the first Chern numbers of
$L^{\pm}_i.$ Then
\begin{eqnarray}\label{norm}
c_1 (N_{C_i |X_{red}})=\alpha^+_i +\alpha^-_i
=\int_{X_{red}}[C_i]^2,~~~i=1,\cdots,k.
\end{eqnarray}
Since we assume $X^{\C^*}$ has no interior isolated fixed points in
this section, i.e., all other connected components in $X^{\C^*}$ are
algebraic curves except the maximum and minimum components.  Then
the diffeomorphism type of Marsden-Weinstein reduced spaces $X_{r}:=X_{red}$ does not depend on $r$ \cite{gs89}
and is a Fano surface. In the following we will detect the
structure of $X$ by analyzing the variation of the Euler class of
circle bundle $P_r\rightarrow X_{red}$ as $r$ varies from the
minimum to the maximum. We denote $e(P_r)$ by $e_-$ if $r$ is near
the minimum and by $e_+$ if $r$ is near the maximum.

\begin{theorem}\label{noiso} Suppose that there is a semifree $\C^*$-action  on a Fano $3$-fold $X$ such that there is no  interior isolated fixed points,
and at least one interior fixed curve (i.e., $a\geq 1$),
then the possible $X$ are  listed Table 6.
\end{theorem}

\begin{table}
\begin{center}
\begin{tabular}{|c ||c |c|c | c|}
\hline $-K^3_X$& $h^{1,1}$ & $h^{1,2}$&a & X
\\ \hline\hline $48$ & 2& 0& 1&$Proj(T_{P^2})$ \\ \cline{2-5}
& 3 & 0& 2&$(P^1)^3$ or $P^1\times\widetilde{P^2}$ \\ \hline
$46$& 2 & 0& 1& blow up of $P^3$ along a line  \\ \cline{2-5}
& 3 & 0& 1&blow up of $P^3$ along a line and a point \\ \hline
$40$& 3 & 0& 1& blow up of $P^1\times P^2$ along a conic on $t\times P^2$  \\ \hline
$24$& 4 & 1& 1& a divisor of degree $(1,1,1,1)$ in $(P^1)^4$ \\ \hline
$36$& 4 & 0& 1&  blow up of $Proj(T_{P^2})$ along two curves  \\ \cline{2-5}
& 5 & 0& 2&   the blow up of $Y$ along two exceptional curve $l,l',$    \\
&  & & & where $Y$ is the blow-up of $P^3$  along two lines   \\ \hline
$38$& 4 & 0& 1&  blow up of $(P^1)^3$ along a tridegree $(0,1,1)$ curve  \\ \hline
$32$& 4 & 0& 1&  blow up of $P^1\times P^2$ at  a $(2,1)$ and a  $(1,0)$ curve  \\ \hline
\end{tabular}
\end{center}
\caption{Case $a\geq 1,b=c=0$}
\end{table}

\begin{proof} By Proposition \ref{maxmin} and Proposition \ref{hodgenumber}, if $F_-$ and $F_+$ are not isolated fixed points then they are Fano manifolds. Moreover the reduced space $X_{red}$ is a del Pezzo surface and will not change its diffeomorphic type.
The del Pezzo surfaces are listed as: $P^2,~P^1\times P^1,$ the blow-up $S_d$ of $P^2$ at $9-d$ points in general positions for $1\leq d\leq 8.$
We will give the proof  according to the types of the Marsden-Weinstein reduced spaces.
Note if the Picard number of $X$ is bigger than $6,$ then $X$ is of form $P^1\times S_d,$  actually we only needed to consider the following $8$ cases:

 {\it Case 1}. Both $F_+$  and $F_-$ are isolated fixed points.  Integrate $(c_1^{S^1}\!(T_X\!))^3\in H^2_{S^1}(X,\Z),$ via using Theorem \ref{ablocal} and the data  in Table 4 we get
$$-K_X^3=54.$$ Note in thise case $X_{red}=P^2.$ Let $u=c_1({\mathscr
O}_{P^2}(1)),$ then $c_1(X_{red})=3u.$ If $r$ is near the minimum (resp.
maximum) then $P_r\rightarrow X_r$ is a Hopf fibration
$S^1\hookrightarrow S^5\rightarrow P^2$ whose Euler class is $e_-
=-u$ (resp. $e_+ =u$)  by
Proposition \ref{localization}. Suppose $C_1,\cdots,C_a$ are all interior fixed
algebraic curves, and suppose $[C_k]=x_k u,$ note that $x_k>0$ since
$C_k$ is an embedded K\"ahler submanifold. Hence
$$-u+x_1 u+x_2 u+\cdots+x_a u =u.$$
 Hence we have $a=1$ or $a=2.$
On the other hand,
 by the adjunction formula
$$\int_{P^2}(x_k u)^2-3\int_{P^2}x_ku^2=2g_k-2,$$
i.e.,
\begin{eqnarray}x_k^2-3x_k+2=2g_k,\quad k=1,2,\cdots,a.\label{xeq}\end{eqnarray}
 By Atiyah and Bott's Lefschetz fixed point formula in Theorem \ref{lef}, via
using the local data computed in Proposition \ref{lefinterior} and
\ref{leflocal} we get
$$\begin{array}{rcl}1=\frac{-t^3}{(1-t)^3}+\frac{1}{(1-t)^3}+\sum_{k=1}^a\frac{-t(1-t)\chi (C_k)-(\alpha^+_k -\alpha^-_k t)t}{(1-t)^3},\end{array}$$
Hence
\begin{eqnarray}\sum_{k=1}^a(\alpha^{+}_k +\chi (C_k))=\sum_{k=1}^a(\alpha^{-}_k +\chi (C_k))=3.\label{reqr}\end{eqnarray}
 We
have following cases:

 {\it Subcase 1}. $a=1.$ By Proposition \ref{hodgenumber} the
Hodge numbers of $X$ are given by $h^{1,1}=1, h^{1,2}=0.$ By (\ref{xeq}) we have
$x_1=2, g_1=0,$ there is only one fixed rational curve $C_1.$ Hence
$\alpha^+_1=\alpha^-_1=2.$
By the classifications of Fano threefolds \cite[Table \S 12.2]{ip}, the only Fano threefold with $h^{1,1}=1,h^{1,2}=0$ and $-K^3_X=54$ is a smooth quadric
$Q\subset P^4.$ Hence we have the following two exact sequences:
 $$0\rw T_X\rw T_{P^4}\rw {\mathscr O}_{P^4}(2)\rw 0,\hskip1cm 0\rw T_{C_1}\rw T_{X}|_{C_1}\rw N_{C_1}X\rw 0$$
Since $ T_{C_1}={\mathscr O}_{P^1}(2)$ and $c_1 (N_{C_1}X)= \alpha^-_1+\alpha^+_1=4,$ by the second exact sequence we have $c_1( T_{X}|_{C_1})=6.$ However, by the Euler exact sequence $0\rw {\mathscr O}_{P^4}\rw {\mathscr O}_{P^4}(1)^{\oplus 5}\rw T_{P^4}\rw 0$ we have $c_1(T_{P^4}|_{C_1})=5$ and hence by the first exact sequence we have $c_1(T_X|_{C_1})=3,$ it is a contradiction, hence this case is impossible.

{\it Subcase 2}. $a=2.$ By Proposition \ref{hodgenumber} the Hodge numbers of $X$ are given
by $h^{1,1}=2, h^{1,2}=0.$  By (\ref{xeq}) we have $x_1=x_2=1$ and $g_1=g_2=0,$ there are exactly two fixed
rational curves $C_1$ and $C_2.$ Note that $\chi(C_1)=\chi(C_2)=1$ and by (\ref{reqr}) $\alpha^{\pm}_1
+\alpha^{\pm}_2=2.$
By the classifications of Fano threefolds \cite[Table \S 12.3]{ip}, the  Fano threefolds with $h^{1,1}=2,h^{1,2}=0$ and $-K^3_X=54$ are
$P^1\times P^2$ and the blow-up of $P^3$  along a line.

The blow-up of $P^3$  along a line is clearly impossible. Firstly, the blowing up center line
is clearly invariant but not point-wise fixed, otherwise there is an induced semifree action on $P^3$ with two isolated fixed points and an interior fixed curve, it is a contradiction to Subcase 1.
Since the induced action of $\C^*$ on the blowing up center line has
exactly two fixed point, after blowing down the exceptional divisor, it will means that there is a semifree $C^*$ action on $P^3$ whose fixed point set is two points, three points or four points, they are all impossible by Corollary \ref{p3}.

The case $X=P^1\times P^2$ is neither possible, otherwise we have exact sequences
$$0\rw T_{C_i}\rw T_{P^1\times P^2}|_{C_i}\rw N_{C_i}\rw 0,\quad\quad i=1,2.$$
Since  $\alpha^{\pm}_1
+\alpha^{\pm}_2=2$ we have $c_1(N_{C_1})+c_1(N_{C_2})=4.$ Therefore $c_1 (T_{P^1\times P^2}|_{C_1})+c_1 (T_{P^1\times P^2}|_{C_2})=8.$
Write $T_{P^1\times P^2}|_{C_i}={\mathscr O}_{P^1}(2)\oplus {\mathscr O}_{P^1}(a_i)\oplus {\mathscr O}_{P^1}(b_{i}),$
by the Euler sequences for $P^2$  we have  $a_i+b_i=c_1(T_{P^2})=3.$
 It is clearly impossible. Hence this subcase is impossible.

To sum up, the case 1 is impossible.

{\it Case 2}. Both $F_+$  and $F_-$ are rational curves.  Then the Marsden-Weinstein reduced space $X_{red}$ is a $P^1$ bundle over
$P^1.$ Since $X_{red}$ is a Fano surface hence  $X_{red}=P^1\times P^1$ or $\widetilde{P^2}\cong P_{P^1}({\mathscr
O}\oplus {\mathscr O}(-1))$, the blowing up of $P^2$ at one point. Both of them are $S^2$ bundle over $S^2.$ Let $u,v$ be the dual class of the
fiber and base respectively.

{\it Subcase 1.} The Marsden-Weinstein reduced space $X_{red}=P^1\times P^1.$ Note
$c_1(X_{red})=2u+2v.$ Let $[C_k]=x_ku+y_kv$ and $e_{\pm
}=x_{\pm}u+y_{\pm}v,$ then $x_k\geq 0,y_k\geq 0$ and $x_k +y_k>0.$
By the adjunction formula
$$\begin{array}{rcl}\int_{P^1\times P^1}[(x_k -2)u+(y_k -2)v](x_k u +y_k v)=2g_k-2,\end{array}$$
i.e.,
\begin{eqnarray}\begin{array}{rcl}(x_k -1)(y_k-1)=g_k,\quad k=1,2,\cdots,a.\end{array}\label{genns}\end{eqnarray}
 Note near the minimum of $J,$ the restriction of circle bundle
$J^{-1}(r)\rightarrow X_{red}$ to the fibre of sphere bundle of
$X_{red}\rightarrow P^1$ is a Hopf fibration $S^1\hookrightarrow
S^3\rightarrow S^2$. We may assume $e_-=-u+y_- v,$ and $e_+=u+y_+v$
or $e_+=x_+ u\pm v.$ On the other hand,  the Euler number of the
normal bundle $N_{F_-|X_{red}}$  is
$$\begin{array}{rcl}c_1 (N_{F_-})=\int_{P^1}u
=\int_{P^1\times P^1}u^2=0,\end{array}$$ and by the same reason we have $c_1
(N_{F_+})=0.$
Integrate $(c_1^{S^1}\!(T_X\!))^3\in H^2_{S^1}(X,\Z),$ via using Theorem \ref{ablocal} and the data  in Table 4 we get
$$-K_X^3=12(c_1(F_-)+c_1(F_+))+4(c_1
(N_{F_-})+c_1
(N_{F_+}))=48.$$
 By   Proposition \ref{localization}, we have
$$\begin{array}{rcl}\int_{P^1\times P^1}e_-^2 =\int_{P^1\times P^1}  e_+^2 =0.\end{array}$$
Hence we have $y_-=x_+ =y_+ =0$ and  thus $e_-=-u$ or $e_+ =u$ or
$\pm v.$

Using the Lefschetz fixed point formula in Theorem \ref{lef} and the
data in Proposition \ref{lefinterior} and \ref{leflocal}, noting that
$c_1(N_{F_+})=c_1(N_{F_-})=0,$ we have
$$\begin{array}{rcl}1=\frac{(1-t)t^2}{(1-t)^3}+\frac{(1-t)}{(1-t)^3}+\sum_{k=1}^a\frac{-t(1-t)\chi (C_k)-(\alpha^+_k
-\alpha^-_k t)t}{(1-t)^3},\end{array}$$ hence
\begin{eqnarray}\begin{array}{rcl}\sum_{k=1}^a(\alpha^+_k+\chi (C_k))=\sum_{k=1}^a(\alpha^-_k+\chi (C_k))=2.\end{array}\label{gesum}\end{eqnarray}

 Since $e_- +\sum_{k=1}^a (x_k u+y_k v)=e_+$ and $x_k\geq 0,y_k\geq 0,g_1\geq 0$ and $x_k +y_k\geq 1.$ We have only the
following two cases:

(i) $a=1.$  If $g_1\geq 1$ then $x_1\geq 2, y_1\geq 2$ which is impossible.
  Hence $g_1=0$ and  $C_1=u+v$ by (\ref{genns}) and $e_- =-u, e_+=v.$ Hence
$h^{1,1}=2, h^{1,2}=0$ and $\alpha^+_1
=\alpha^-_1 =1.$ By the classifications of Fano threefolds \cite[Table \S 12.3]{ip}, the only Fano threefold with $h^{1,1}=2,h^{1,2}=0$ and $-K^3_X=48$ is
$Proj(T_{P^2}).$

 (ii) $a=2.$ Then $C_1=C_2=u$ or $C_i=u,
C_{3-i}=v$ for $i=1,2.$ We have $g_1=g_2=0$ by (\ref{genns}) and $h^{1,1}=3, h^{1,2}=0.$ Suppose now $J(C_1)\not=J(C_2).$
 By Proposition \ref{localization},$\alpha^+_1
=\alpha^-_1$  and $\alpha^+_2 =\alpha^-_2.$
By (\ref{gesum}) we have $\alpha^+_1
+\alpha^+_2=\alpha^-_1 +\alpha^-_2=0.$
  Hence we have $\alpha^+_1
=\alpha^-_1=\alpha^+_2 =\alpha^-_2=0.$
By the classifications of Fano threefolds \cite[Table \S 12.4]{ip}, there are two Fano threefold with $h^{1,1}=3,h^{1,2}=0$ and $-K^3_X=48,$ they
are $P^1\times P^1\times P^1$ and $P^1\times \widetilde{P^2}.$ On $P^1\times P^1\times P^1$ such semifree $\C^*$-action is realized by
$$t\cdot ([x_1,x_2],[y_1,y_2],[z_1,z_2])=([tx_1,x_2],[ty_1,y_2],[z_1,z_2]).$$
On $P^1\times \widetilde{P^2}.$
such semifree $\C^*$-action is defined as the following: Let $\C^*$ act on $P^2$ by
$$t\cdot ([z_1,z_2,z_3])=([tz_1,z_2,z_3]),$$
blowing up the fixed point we get an semifree action on $\widetilde{P^2}$ with two fixed lines. Let $\C^*$ act on $P^1$ by $t\cdot ([x_1,x_2])=([tx_1,x_2).$ Then there is a natural product action on $P^1\times \widetilde{P^2}$ which is semifree and with required properties.

{\it Subcase 2.}  The Marsden-Weinstein reduced space $X_{red}=\widetilde{P^2}\cong P_{P^1}({\mathscr
O}\oplus {\mathscr O}(-1)).$ Let $u,v\in H^2(\widetilde{P^2},\Z)$ be a basis which
are dual classes of the fibres and the zero section $P^1$
respectively. Then
$$\begin{array}{rcl}u^2=0,~~~~v^2=-1, ~~~~~uv=1.\end{array}$$
Let $e_{\pm} =x_{\pm}u+y_{\pm} v.$ Since the restriction of circle
bundle $\varphi^{-1}(a)\rightarrow X_{red}$ on the fibre of
$X_{red}\rightarrow P^1$ is a Hopf fibration,
 we have$$\begin{array}{rcl}
c_1 (N_{F_-
})&=&\int_{fibre}e_-=(x_-u+y_-v)u=-1,\\
 c_1 (N_{F_+})&=&\int_{fibre}e_+=(x_+u+y_+v)u=1.\end{array}$$
 Therefore $y_-=-1$ and $y_+=1.$
 Integrate $(c_1^{S^1}\!(T_X\!))^3\in H^2_{S^1}(X,\Z),$ via using Theorem \ref{ablocal} and the data  in Table 4 we get
$$-K_X^3=12(c_1(F_-)+c_1(F_+))+4(c_1
(N_{F_-})+c_1
(N_{F_+}))=48.$$ By  Proposition \ref{localization},
$\int_{\widetilde{P^2}}e_-^2 =1,~~\int_{\widetilde{P^2}}  e_+^2 =-1.$
Thus $x_-=0,x_+=1$ and
\begin{eqnarray}\sum_{k=1}^a x_k=1,\quad \sum_{k=1}^a y_k =2.\label{sbb0}\end{eqnarray}
Using  the Lefschetz fixed point formula in Theorem \ref{lef} we have
$$\begin{array}{rcl}1=\frac{-t^3}{(1-t)^3}+\frac{1}{(1-t)^3}+\sum_{k=1}^a\frac{-t(1-t)\chi(C_k)-(\alpha^+_k
-\alpha^-_k t)t}{(1-t)^3},\end{array}$$ hence
\begin{eqnarray}\sum_{k=1}^a(\alpha^+_k+\chi (C_k))=\sum_{k=1}^a(\alpha^-_k+\chi (C_k))=3.\label{sbb1}\end{eqnarray}
By formula (\ref{norm}), we have
\begin{eqnarray}2 x_k y_k-y^2_k =\alpha^+_k +  \alpha^-_k,~~~k=1,2,\cdots,a.\label{sbb2}\end{eqnarray}
On the other hand, by the adjunction formula
$$\begin{array}{rcl}\int_{\widetilde{P^2}}[(x_k u +y_k v)-(3u+2v)](x_k u +y_k v)=2g_k-2,\end{array}$$
i.e.,
\begin{eqnarray}(2x_k -1)(y_k-1)+1=y^2_k+2g_k ,~~~k=1,2,\cdots,a.\label{sbb4}\end{eqnarray}
From (\ref{sbb4}) we know if $x_k=0$ then $y_k=1, g_k =0;$ if $y_k=0$ then $x_k =1,
g_k=0.$ By(\ref{sbb0}) in the rest
cases we have $1\leq x_k, y_k\leq 2,$ and then from (\ref{sbb4})  we have $(x_k,y_k,g_k)=(1,1,0),(2,2,0) $ or $(2,1,0).$
 So $g_k =0$ in every possible case. Hence by (\ref{sbb1}) and (\ref{sbb2}) we have \begin{eqnarray}\sum^a_{k=1}(2x_k y_k -y^2_k)=6-2\sum_{k=1}^a\chi(C_k)=6-2a\label{sbbb}\end{eqnarray}
 Take sums of the equations in (\ref{sbb4}), and using  (\ref{sbb0}) and (\ref{sbbb}) we have
    $$2-2a=\sum^a_{k=1}[(2x_k y_k -y^2_k)-y_k -2x_k]=\sum_{k=1}^a(2g_k-2)=-2a,$$
    which is a contradiction.
 Hence this subcase is impossible.

To sum up, in Case 2,  $X$ are $(P^1)^3,P^1\times \widetilde{P^2}$ or $Proj(T_{P^2}).$

{\it Case 3.}  $F_+$ and $F_-$ are $P^2$ and an isolated fixed point. Without loss of generality we may assume that $F_+ =P^2$ and $F_- $ is an isolated fixed point. Note the Marsden-Weinstein reduced space $X_{red}=P^2,$ the Hodge number $h^{1,1}=1+a$ and
$h^{1,2}=\sum g_k$ by Proposition \ref{hodgenumber}. Let $[C_k]=x_k u$ with $x_k\geq 1.$
By  Proposition \ref{localization} we may assume $e_- =\pm u.$ Let $e_+ = x_+ u,$
then $c_1(F_+)c_1(N_{F_+})
=3x_+$, and $c_1^2(N_{F_+})=x^2_+.$
Using the Lefschetz fixed point formula in Theorem \ref{lef}
$$\begin{array}{rcl}&&\frac{-t^3}{(1-t)^3}+\sum_{k=1}^a\frac{-t(1-t)\chi (C_k)-(\alpha^+_k -\alpha^-_k
t)t}{(1-t)^3}\\
&&+\frac{2(1-t)^2-3t(1-t)x_++t(1+t)
x^2_+}{2(1-t)^3}  =1\\
\end{array}$$
 we get
\begin{eqnarray}\left\{\begin{array}{rcl}
&&x^2_+ -3x_+ -2\sum_{k=1}^a(\alpha^+_k +\chi (C_k))=-2 ;\hskip2cm\hfill{(3_a)}\\
&&x^2_+ + 3x_+ + 2\sum_{j=1}^a(\alpha^-_k+\chi (C_k)) =4.\hskip2cm\hfill{(3_b)}\\
\end{array}\right.\label{qqn}\end{eqnarray}
By formula (\ref{norm}), $\alpha^+_k +\alpha^-_k =x_k^2.$ Note in (\ref{qqn}) the equation ${\rm (3_b)-(3_a)}$ is
\begin{eqnarray}3x_++\sum_{k=1}^a (x^2_k+2\chi (C_k))   =3.\label{3aa}\end{eqnarray}
On the other hand,
 by the adjunction formula
we have
\begin{eqnarray}x_k^2-3x_k+2\chi(C_k)=0,\quad k=1,2,\cdots,a.\label{3bb}\end{eqnarray}
By (\ref{3aa}) and (\ref{3bb}) we have
\begin{eqnarray}x_++\sum_{k=1}^a x_k=1.\end{eqnarray}
Note $x_+ >x_-\geq -1$ and $x_k \geq 1.$ Hence we have $e_-=-u,
x_+=0$ and $a=1.$  Integrate $(c_1^{S^1}\!(T_X\!))^3\in H^2_{S^1}(X,\Z),$ via using Theorem \ref{ablocal} and the data  in Table 4 we get
\begin{eqnarray}\quad\quad-K_X^3=27+3c_1^2(F_+)+c^2_1
(N_{F_+}))+3c_1(F_+)c_1
(N_{F_+}))=54+x^2_+ +9x_+,\label{3cc}\end{eqnarray}
hence $-K_X^3=54.$
 By Proposition \ref{hodgenumber} the
Hodge numbers of $X$ are given by $h^{1,1}=1+a=2.$ By (\ref{3bb}) we have $h^{1,2}=g_1=0.$
 By the classifications of Fano threefolds \cite[Table \S 12.3]{ip}, $X$ is $P^1\times P^2$ or the blow-up of $P^3$ along a line.

For $X=P^1\times P^2,$ consider the exact sequence
$$0\rw T_{C_1}\rw T_{P^1\times P^2}|_{C_1}\rw N_{C_1} X\rw 0,$$
since $c_1(N_{C_1} X)=\alpha^+_1+\alpha^-_1=1$ we have $c_1( T_{P^1\times P^2}|_{C_1})=3,$ therefore $C_1\subset P^2.$ Hence $P^1\times (P^2-C_1)=P^1\times \C^2$ is $\C^*$-invariant and $P^2\subset P^1\times \C^2$ is a $\C^*$-fixed complex submanifold. This is clearly impossible since $T_{P^2}\nsubseteqq T_{P^1}\times \C^2.$

To sum up, in Case 3, $X$ is  the blow-up of $P^3$ along a line and such a semifree action is given in Example \ref{exppaper1}.

{\it Case 4}. $F_+$ and $F_-$ are $P^1\times P^1$ and a rational curve. Without loss of generality we may assume that $F_+ =P^1\times P^1, F_- =P^1.$ Note the Marsden-Weinstein reduced space $X_{red}=P^1\times P^1.$
 We may assume
 $e_- =\pm u+y_- v.$ From Proposition \ref{localization} we have $$\int_{P^1\times P^1}e^2_{-}=c_1(N_{F_-})=0,$$ hence $y_- =0$ and $e_- =\pm u.$ Therefore $y_+\geq 0.$
  Let $[C_k]=x_k u +y_k v$ and $e_+ =x_+ u +y_+ v.$ Then $x_k\geq 0,y_{k}\geq 0$ and $c_1(F_+)c_1(N_{F_+})=(2u+2v)(x_+ u +y_+ v)=2(x_+ +y_+)$ and $c_1^2(N_{F_+})=2x_+ y_+. $
By Lefschetz fixed point formula in Theorem \ref{lef}
$$\begin{array}{rcl}&&\frac{(1-t)t^2}{(1-t)^3}+\sum_{k=1}^a\frac{-t(1-t)\chi (C_k))-(\alpha^+_k -\alpha^-_k
t)t}{(1-t)^3}\\
&&+\frac{(1-t)^2-t(1-t)(x_+ +y_+)+t(1+t)
x_+ y_+}{(1-t)^3}  =1\\
\end{array}$$
 we get
\begin{eqnarray}\quad\left\{\begin{array}{rcl}
&&x_+ y_+ -x_+ -y_+ -\sum_{k=1}^a(\alpha^+_k+\chi (C_k)))=1;\hskip2cm\hfill{(4_a)}\\
&&x_+ y_+ + x_+ +y_+ +\sum_{k=1}^a(\alpha^-_k+\chi (C_k))) =1.\hskip2cm\hfill{(4_b)}\\
\end{array}\right.\label{case4}\end{eqnarray}
Hence $x_+ y_+=\sum_{k=1}^a(\alpha^+_k -\alpha^-_k).$ On the other hand, by formula (\ref{norm}) we have $\alpha^+_k +\alpha^-_k =2x_k y_k.$  In (\ref{case4}), the equation ${\rm (4_b)-(4_a)}$ is
\begin{eqnarray} x_+ +y_+ +\sum_{k=1}^a (x_k y_k+\chi(C_k))=0.\label{4a}\end{eqnarray}
By the adjunction formula
\begin{eqnarray}x_k y_k-x_k-y_k+\chi (C_k)=0,\quad k=1,2,\cdots,a.\label{4b}\end{eqnarray}
By (\ref{4a}) and (\ref{4b}) we have
\begin{eqnarray} x_+ +y_+ +\sum_{k=1}^a (x_k+ y_k)=0.\label{4c}\end{eqnarray}
Hence we must have $e_-=-u$ and $x_+\geq 0$ and  thus $\sum_{k=1}^a (x_k+y_k)=0.$  It is clearly impossible since $x_k+y_k\geq 1.$

Hence the Case 4 is impossible.

{\it Case 5}. $F_+$ and $F_-$ are $\widetilde{P^2}$ and a rational curve. Without loss of generality we may assume that $F_+ =\widetilde{P^2}, F_- =P^1.$ Note the Marsden-Weinstein reduced space $X_{red}=\widetilde{P^2}.$ Let $u,v\in H^2(\widetilde{P}^2,\Z)$ be a basis which
are dual classes of the fibres and the zero section $P^1$
respectively. Then
$\int_{\widetilde{P^2}}u^2=0,~~\int_{\widetilde{P^2}}v^2=\int_{P^1}
v=-1, $ and $\int_{\widetilde{P^2}}uv=1.$ The restriction of circle
bundle $\varphi^{-1}(a)\rightarrow X_{red}$ on the fibre of
$X_{red}\rightarrow P^1$ is a Hopf fibration with Euler class $v.$
We may assume $e_-=\pm u+x_- v.$ On the other hand,
by formula (\ref{norm}),
$c_1 (N_{F_-
})=\int_{{P}}v=-1,$
  Note they are exactly the Euler numbers of the
normal bundles $N_{F_{\pm}|X_{red}},$ hence by Proposition \ref{localization},
$\int_{\widetilde{P^2}}e_-^2  =1.$
Hence $x_-=\pm 1.$  Thus $e_-=-u-v$ or $e_- =u+v.$
  Let $[C_k]=x_k u +y_k v$ and $e_+ =x_+ u +y_+ v.$  Note $c_1(F_+)=3u+2v.$ So $c_1(F_+)c_1(N_{F_+})=(3u+2v)(x_+ u +y_+ v)=2x_+ +y_+ $ and $c_1^2(N_{F_+})=2x_+ y_+ -y^2_+. $
$$\begin{array}{rcl}&&\frac{(1-t)t^2-t^2}{(1-t)^3}+\sum_{k=1}^a\frac{-t(1-t)\chi (C_k)-(\alpha^+_k -\alpha^-_k
t)t}{(1-t)^3}\\
&&+\frac{2(1-t)^2-t(1-t)(2x_+ +y_+)+t(1+t)
(2x_+ y_+ -y^2_+)}{2(1-t)^3}  =1\\
\end{array}$$
 we get
\begin{eqnarray}\hskip 0.7cm \left\{\begin{array}{rcl}
&&\!\!\!\!\!(2x_+ y_+ -y^2_+)  -(2x_+ +y_+) -2\sum_{k=1}^a(\alpha^+_k +\chi (C_k))=-2; \hskip0.5cm\hfill{(5_a)}\\
&&\!\!\!\!\!(2x_+ y_+ -y^2_+) +(2x_+ +y_+) +2\sum_{k=1}^a(\alpha^-_k +\chi (C_k)) =4.\hskip0.5cm\hfill{(5_b)}\\
\end{array}\right.\label{qbn}\end{eqnarray}
In (\ref{qbn}), the equation ${\rm (5_b)-(5_a)}$ is
\begin{eqnarray}\begin{array}{rcl}
  (2x_+ +y_+) +\sum_{k=1}^a(\alpha^+_k+\alpha^-_k +2\chi (C_k))=3.
   \end{array}\label{zw}\end{eqnarray}
 From the formula (\ref{norm}) we have
 \begin{eqnarray}2 x_k y_k-y^2_k =\alpha^+_i +  \alpha^-_i,~~~k=1,2,\cdots,a.\label{5aa}\end{eqnarray}
On the other hand, by the adjunction formula for curves $C_k$ we have
 \begin{eqnarray}2x_ky_k-y^2_k -2x_k -y_k+2\chi(C_k)=0,~~~k=1,2,\cdots,a.\label{5bb}\end{eqnarray}
 By (\ref{5aa}) and (\ref{5bb}) we have
\begin{eqnarray} 2x_+ +y_+ +\sum_{k=1}^a (2x_k+ y_k)=3.\label{5cc}\end{eqnarray}
Hence we must have $e_-=-u-v$ and $x_-=y_-=-1.$ Since $x_+=\sum x_k-1$ and  $y_+=\sum y_k-1,$ we have
\begin{eqnarray} \sum_{k=1}^a (2x_k+ y_k)=3.\label{5dd}\end{eqnarray}
Hence $1\leq a\leq 3.$

{\it Subcase 1.} $a=1.$ Then $x_1=y_1=1$ or $x_1=0, y_1=3.$ The genus of interior fixed curve $C_1$ is $g_1=0.$
The
Hodge numbers of $X$ are given by $h^{1,1}=2+a=3$ and $h^{1,2}=0.$
Integrate $
(c_1^{S^1}\!(T_X\!))^3\in H^2_{S^1}(X,\Z),$ via using Theorem \ref{ablocal} and the data  in Table 4 we get
$$\begin{array}{rcl}\quad\quad-K_X^3&=&12c_1(F_-)+4c_1
(N_{F_-})+c^2_1
(N_{F_+}))+3c_1(F_+)c_1
(N_{F_+})\\
&=&20+2x_+ y_+-y^2_++3(2x_++y_+).\\\end{array}$$
If  $x_1=y_1=1$  then $-K_X^3=20.$
  If  $x_1=0,y_1=3$  then $x_+=-1$ and $y_+=2,$ hence  $-K_X^3=12.$
By the classifications of Fano threefolds  with $h^{1,1}=3$ in \cite[Table \S 12.4]{ip}, this subcase does not exist.

{\it Subcase 2.} $a=2.$ Without loss generality, we assume $x_1\geq x_2.$ If $x_1=1$ then $x_2=0$ and $y_1+y_2=1.$ Hence $x_+=0$ and $y_+=1$
and we have
$$-K_X^3=20+2x_+ y_+-y^2_++3(2x_++y_+)=22.$$
The
 Picard number of $X$ are given by $h^{1,1}=2+a=4.$  By the classifications of Fano threefolds  with $h^{1,1}=4$ in \cite[Table \S 12.5]{ip} and \cite{mm2}, this subcase does not exist.

{\it Subcase 2.} $a=3.$ Then $x_1=x_2=x_3=0$ and $y_1=y_2=y_3=1.$ Hence $x_+=-1$ and $y_+=2$
and we have
$$-K_X^3=20+2x_+ y_+-y^2_++3(2x_++y_+)=12.$$
The Picard number of $X$ are given by $h^{1,1}=2+a=5.$  By the classifications of Fano threefolds  with $h^{1,1}=5$ in \cite[Table \S 12.6]{ip}, this subcase does not exist.

To sum up, the Case 5 is impossible.

\vskip 1cm

In the rest of this proof, we consider the cases where both $F_+$ and $F_-$ are Fano surfaces. Suppose $C_1,\cdots,C_m$ are all fixed curves.
By Lefschetz fixed point formula in Theorem \ref{lef}
$$\begin{array}{rcl}\sum_{i=1}^m\!\!\frac{-t(1-t)\!\chi(C_1\!)-(\alpha^+_1\!-\!\alpha^-_1 t\!)t}{(1-t)^3}+\frac{-2t(1-t)^2-t(1-t)c_1(F_+)c_1(N_{F_+})-t(1+t)c_1^2(N_{F_+}))}{2(1-t)^3}&&\\
+\frac{2(1-t)^3-t(1-t)c_1(F_-)c_1(N_{F_-})+t(1+t)
c_1^2(N_{F_-})}{2(1-t)^3}  &=&1\\
\end{array}$$
Comparing the coefficients of $t^k$ for $k=0,1,2,3$ we get
$$\left\{\begin{array}{rcl}
\!\!\!(\!c_1^2\!(N_{F_+}\!)\!-\!c_1^2\!(N_{F_-}\!)\! )\! -\!(c_1(F_+\!)c_1\!(N_{F_+}\!)\!+\!c_1\!(F_-\!)c_1\!(N_{F_-}\!)\! ) \!-\!2\sum_{i=1}^m\!(\alpha^+_1+\chi \!(C_1)\!)\!\!\!&=&\!\!\!0 ;\quad\hfill{(a)}\\
\!\!\!(\!c_1^2\!(N_{F_+}\!)\!-\!c_1^2\!(N_{F_-}\!)\! )\! +\!(c_1(F_+\!)c_1\!(N_{F_+}\!)\!+\!c_1\!(F_-\!)c_1\!(N_{F_-}\!)\! )\! +\!2\sum_{i=1}^m\!(\alpha^+_1+\chi \!(C_1)\!)\!\!\!&=&\!\!\!0.\quad\hfill{(b)}\\
\end{array}\right.$$
By (a)+(b) and (a)-(b), we have
\begin{eqnarray}\quad\left\{\begin{array}{rcl}
\!\!\!c_1^2(N_{F_+})-c_1^2(N_{F_-}) &=&\sum_i\!(\alpha^+_i\!-\!\alpha^-_i\!);\quad\hfill{(c)}\\
\!\!\!(c_1(F_+)c_1(N_{F_+})+c_1(F_-)c_1(N_{F_-}) )&=& -\sum_i\!(\!\alpha^+_i\!+\!\alpha^-_i\!+\!2\chi\! (C_1\!)\!).\quad\hfill{(d)}
\end{array}\right.\label{surff}
\end{eqnarray}
Integrate $(c_1^{S^1}\!(T_X\!))^3\in H^2_{S^1}(X,\Z),$ via using Theorem \ref{ablocal} together with Proposition \ref{rinterior} and \ref{rlocal} we get
\begin{equation}\begin{array}{rcl}
-K_X^3&=&3(c^2_1({F_{+}})+c^2_1({F_{-}}))+c^2_1(N_{F_{+}})+c^2_1(N_{F_{+}})\\
&&+3(c_1({F_{+}})c_1(N_{F_{+}})+c_1({F_{-}})c_1(N_{F_{-}}))\\
\end{array}\label{can3}\end{equation}

{\it Case 6.}
Both $F_+ $ and $F_-$ are $P^2.$ Note the Marsden-Weinstein reduced space $X_{red}=P^2.$
Assume $e_{\pm}=x_{\pm} u$ and $[C_k]=x_k u .
$   Note $x_k >0$ and $c_1(F_{\pm})=3u.$ So $c_1(F_{\pm})c_1(N_{F_{\pm}})=3x_{\pm}$ and $c^2_1(N_{F_{\pm}})=x^2_{\pm}.$
By the adjunction formula
\begin{eqnarray}(x_k -3)x_k=2g_k-2 ,~~~k=1,2,\cdots,a.\label{dr1}\end{eqnarray}
Since $h^{1,1}=2+a,$ if $a>3$ then $X=P^1\times S_d$ and it very easy to check it is impossible.
Hence it suffices to consider $a=1,2$ and $a=3.$

{\it Subcase 1.} $a=1.$ In this subcase we have $h^{1,1}=3$ and $h^{1,2}=g_1.$  By (d) of (\ref{surff}) we have
\begin{eqnarray}3x_+ +3x_-  = -(\alpha^+_1+\alpha^-_1+2\chi (C_1))=-x_1^2+2g_1-2,\label{dr2}\end{eqnarray}
hence $6x_+ -3x_1 =-x_1^2+2g_1-2.$  By (\ref{dr1}) we must have $x_+=0.$ By (c) of (\ref{surff}) we have
$\alpha^+_1-\alpha^-_1=-x_1^2,$ since $\alpha^+_1+\alpha^-_1=x_1^2$ we must have $\alpha^+_1=0,\alpha^-_1=x_1^2.$
By (\ref{can3}) we have
$$-K^3_X=54+x^2_++x^2_- +9(x_+ +x_-)=54+x^2_1-9x_1 $$
By the classifications of Fano threefolds \cite[Table \S 12.4]{ip}, the Fano threefolds with $h^{1,1}=3$ must have $g_1=h^{1,2}\leq 8,$ hence
$x^2_1-3x_1\leq 14$ and thus $1\leq x_1 \leq 6.$ Using
 the classifications in \cite[Table \S 12.4]{ip}, case by case check, it is easy to see only $x_1=1,2$ are possible and when $x_1=1$ we have $g_1=0$
and $-K^3_X=46,$ when $x_1=2$ we have $g_1=0$
and $-K^3_X=40.$

{\it Subcase 2.} $a=2.$ In this subcase we have $h^{1,1}=4$ and $h^{1,2}=g_1+g_2.$  By (d) of (\ref{surff}) we have
\begin{eqnarray}3x_+ +3x_-  = -\sum_{i=1}^2(\alpha^+_i+\alpha^-_i+2\chi (C_i))=-x_1^2-x^2_2+2g_1+2g_2-4,\label{dr2}\end{eqnarray}
hence $6x_+ -3x_1-3x_2 =-x_1^2-x_2^2+2g_1+2g_2-4.$   By (\ref{dr1}) we must have $x_+=0.$
Therefore (\ref{can3}) we have
$$-K^3_X=54+x^2_++x^2_- +9(x_+ +x_-)=54+(x_1+x_2)^2-9(x_1+x_2) $$
By the classifications of Fano threefolds \cite[Table \S 12.5]{ip}, the Fano threefolds with $h^{1,1}=4$ must have $h^{1,2}=0$ or $1.$

If $h^{1,2}=1,$ then $g_1+g_2=1.$ Without loss of generality we assume $g_1=0$ and $g_2=1.$ Then $x_1=1$ or $2$ and $x_2=3.$ Only
$(x_1,x_2)=(1,3)$ and $(2,3)$ are possible, and in these cases $-K^3_X=24.$  By the classifications in \cite[Table \S 12.5]{ip} and \cite{mm2}, $X$ is a divisor on $P^1\times P^1\times P^1\times P^1$ of multi-degree $(1,1,1,1).$ Hence we have the following two exact sequences:
 $$0\rw T_X\rw T_{(P^1)^4}\rw {\mathscr O}_{(P^1)^4}(1,1,1,1)\rw 0,\quad 0\rw T_{C_1}\rw T_{X}|_{C_1}\rw N_{C_1}X\rw 0$$
Since  $c_1(N_{C_1}X)= x^2_1,$ by the second exact sequence we have $c_1( T_{X}|_{C_1})=2+x^2_1>2.$ However by the first exact sequence we have $c_1(T_X|_{C_1})\leq 2$ since  $T_{(P^1)^4}= {\mathscr O}_{P^1}(2)\oplus {\mathscr O}_{P^1}(2)\oplus {\mathscr O}_{P^1}(2)\oplus{\mathscr O}_{P^1}(2),$
 it is a contradiction, hence this case is impossible.

If $h^{1,2}=0,$ then $g_1=g_2=0.$  Then $x_1$ and $x_2$ are equal $0$ or $3.$ In these cases we have $-K^3_X=36.$ By the classifications in \cite[Table \S 12.5]{ip} and \cite{mm2}, $X$ is the blow-up of $Proj(T_{P^2})$ along a disjoint union of two curves.

{\it Subcase 3.} $a=3.$ Just as in Subcase 2, we have $x_+=0$ and
$$-K^3_X=54+(x_1+x_2+x_3)^2-9(x_1+x_2+x_3). $$
Since $h^{1,1}= 5$ we have $h^{1,2}=0$ and thus $g_1=g_2=g_3=0.$ Hence $x_k=1$ or $2$  for $1\leq k\leq 3$ and
hence $-K^3_X=36$ or $24.$  By the classifications in \cite[Table \S 12.6]{ip}, the possible $X$ is the blow-up of $Y$ along two exceptional lines of the
blowing up $Y\rightarrow P^3$ along two lines. This is impossible, otherwise there is a semifree $\C^*$-action on $Y$ whose maximum and minimum component are $P^2$ and whose interior fixed points is a rational curve, it is impossible by analysis in Subcase 1.
Hence the Subcase 3 is impossible.

To sum up in the Case 6 we have only the following possibilities:
  \begin{enumerate}
\renewcommand{\labelenumi}{(\arabic{enumi})}
\item $-K_X^3=40,$ the Hodge number $h^{1,1}=3$ and $h^{1,2}=0.$ $X$ is the blow up of $P^1\times P^2$ along a conic on $t\times P^2.$
\item $-K_X^3=46,$ the Hodge number $h^{1,1}=3$ and $h^{1,2}=0.$ $X$ is the blow up of $P^3$ along a disjoint union of a line and a point.
\item $-K_X^3=36,$ the Hodge number $h^{1,1}=4$ and $h^{1,2}=0.$ $X$ is the blow-up of  $Proj(T_{P^2})$ along two curves.
\end{enumerate}

\begin{lemma}\label{qqr} There is no semifree $\C^*$ action on the smooth quadric $Q\subset P^4.$
\end{lemma}
\begin{proof}
Since the Picard number of $Q$ is $1,$ the connected component of the fixed point set where the moment map takes  maximum or minimum is a line or an isolated point by Proposition \ref{hodgenumber}. If both of them are isolated point then there is exactly one interior fixed rational curve, which is impossible by analysis in Case 1. If both of them are lines then there is no interior fixed points and it is neither possible by our classification in Table 5. If one of then is a line and the other is an isolated point, without loss of generality suppose $F_+$ is a rational curve and $F_-$ is a point.
Hence $b=1$ and $c=0.$ Integrate $1\in H^*_{S^1}(Q,\Z)$  we get $c_(N_{F_+})=0;$ integrate $(c_1^{S^1}\!(T_X\!))^3\in H^2_{S^1}(X,\Z)$ we have
 $$\begin{array}{rcl}
-K_Q^3=12c_1({F_{+}})+4c_1(N_{F_{+}})+27-b-c=50,
\end{array}$$
however it is well known that $-K_Q^3=54.$
\end{proof}

{\it Case 7.}
 Both $F_+ $ and $F_-$ are $P^1\times P^1.$  Note the Marsden-Weinstein reduced space $X_{red}=P^1\times P^1$ and $h^{1,1}=3+a.$
 If $a\geq 2,$ by the classifications of Fano threefolds in \cite[Table \S 12.6]{ip}, $-K^3_X=28$ or $36,$ or $X$ is of form $P^1\times S_k$ with $k\leq 6,$
where $S_k$ is the blow up of $P^2$ at $9-k$ distinct general points. The last case is clearly impossible since there is no semifree  $\C^*$ action on $P^2$ whose fixed point set consisting of more than two isolated fixed points and a fixed line. If $-K^3_X=28$ then  $X$ is the blow-up of $Y$ along the three exceptional lines of the blowing up $Y\rw Q,$ and $Y$ is the blow-up of the smooth quadric $Q$ along a conic. It is possible, otherwise we will obtain a semifree $\C^*$-action on the quadric $Q$ after blowing down, which will contradict to Lemma \ref{qqr}.
Hence the only possibility is  $-K^3_X=36$ if $a\geq 2$

In the following proof of Case 7,  we  assume $a=1.$

  Let $u,v\in H^2(P^1\times P^1,\Z)$ be a basis which
are dual classes of the fibres and the zero section $P^1$
respectively. Assume $e_{\pm}=x_{\pm} u+y_{\pm}v,$ and $[C_1]=x_1 u +y_1 v.$
  Note $c_1(F_{\pm})=2u+2v.$ So $c_1(F_{\pm})c_1(N_{F_{\pm}})=2x_{\pm}+2y_{\pm}$ and $c^2_1(N_{F_{\pm}})=2x_{\pm}y_{\pm}.$
 On the other hand, by the adjunction formula
\begin{eqnarray}x_1 y_1-x_1-y_1=-\chi (C_1).\label{wd5}\end{eqnarray}
By the classifications of Fano threefolds \cite[Table \S 12.5]{ip} and \cite{mm2}, the Fano threefolds with $h^{1,1}=4$ must have $24\leq -K^3_X\leq 46$ and $h^{1,2}=1$ or $h^{1,2}=0.$ Note $c_1^2(F_+)=c_1^2(F_-)=8,$ by (\ref{can3}) we have
\begin{eqnarray}-K_X^3=48 +2(x_+y_++x_-y_-)+6(x_+ +y_++x_- +y_- ).\label{dp1}\end{eqnarray}

{\it Subcase 1.} $h^{1,2}=1.$ Then by  (\ref{wd5}) we have $(x_1-1)(y_1-1)=1,$ since $x_1\geq 0,y_1\geq 0 $ and $x_1+y_1\geq 1$ we must have
$x_1=y_1=2.$ Since $\alpha^+_1+\alpha^-_1=2x_1 y_1,$  by (d) of (\ref{surff}) we have
\begin{eqnarray}2(x_+ +y_++x_- +y_-) = -(\alpha^+_1+\alpha^-_1+2\chi (C_1))=-8\label{dp2}\end{eqnarray}
hence $x_+ +y_+=0.$ by (\ref{dp1}) and (\ref{dp2}) we have
$$
-K_X^3=24-2x_+^2 -2(x_+-2)(x_++2)
=32-4 x_+^2$$
hence  $-K_X^3=24$ or $28.$

 If $-K_X^3=28,$ by the classifications of Fano threefolds \cite[Table \S 12.5]{ip} and \cite{mm2}, $X$ is the blow up of the cone over a smooth quadric $S\subset P^3$ along a disjoint union of the vertex and an elliptic curve on $S.$  Clearly both exceptional divisors are not point-wise fixed. Since there is no interior isolated fixed points  we may assume there are at least one fixed
curve  on each exceptional divisors which are not the blowing up center by Proposition \ref{blowu}. Hence the $\C^*$-action on $X$
has at least $2$ interior fixed curves, which is a contradiction.

{\it Subcase 2.} $h^{1,2}=0.$ Then by  (\ref{wd5}) we have $(x_1-1)(y_1-1)=0.$ Without loss of generality we assume $x_1=1.$ Then  $\alpha^+_1+\alpha^-_1=2 y_1\geq 0.$  By (d) of (\ref{surff}) we have
\begin{eqnarray}2(x_+ +y_++x_- +y_-) = -(\alpha^+_1+\alpha^-_1+2\chi (C_1))=-2(1+y_1)\label{dp3}\end{eqnarray}
hence
 \begin{eqnarray}x_+ +y_+=0.\label{dpp}\end{eqnarray}
 By (\ref{dp1}) and (\ref{dp2}) we have
$$\begin{array}{rcl} -K_X^3&=&48- 2x_+^2+2(x_+-1)(-x_+-y_1)-6(1+y_1 )\\
&=&42-4x^2_++2x_+(1-y_1)-4y_1\\
&=&-3x^2_+-[x_+-(1-y_1)]^2+ (y_1^2-6y_1+43).\\
\end{array}$$
By the classifications of Fano threefolds \cite[Table \S 12.5]{ip} and \cite{mm2}, if $h^{1,1}=4$ and $h^{1,2}=0$ then
$26\leq 42-4x^2_++2x_+(1-y_1)-4y_1\leq 46,$ equivalently
  \begin{eqnarray}(-4x^2_++2x_+)-4\leq y_1(2x_++4)\leq (-4x^2_++2x_+)+16,\label{dp9}\end{eqnarray}
  in particular $x_+\not=-2$ since if $x_+=-2$ we have $-K^3_X=22.$ In the following we will consider cases $x_+>-2$ and $x_+<-2$
respectively.

If $x_+<-2,$ since $\frac{-4x^2_++2x_+}{2x_++4}=-2x_++5-\frac{20}{2x_++4},$
 by (\ref{dp9}) we have
\begin{eqnarray}2x_++5-\frac{24}{2x_++4}\geq y_1\geq 2x_++5-\frac{4}{2x_++4},\label{dp10}\end{eqnarray}
Note if $x_+\leq -5$ then $2x_++5-\frac{24}{2x_++4}\leq 0,$ contradict to that $y_1\geq 0.$ Hence $x_+\geq -4.$
If $x_+=-3$ then $-K^3_X=2y_1$
If $x_+=-4$ then $-K^3_X=4y_1-30.$
Since $-K^3_X\geq 30,$ we must have $y_1\geq 15.$ Hence $c_1(N_{C_1}X)=2x_1y_1\geq 30.$ However, by the classifications in \cite[Table \S 12.5]{ip} and \cite{mm2},
each $X$ is obtained from several blowing-ups  from $P^1\times P^1\times P^1, P^1\times S, P^3$ or  obtained from several blowing-ups from divisors (with small bidegree $(1,1)$) in
$P^2\times P^2,$ hence $c_1(T_X|_{C_1})\leq 6.$ (We give an example to indicate it. For example, the blow up of $P^3$ along two disjoint lines, after the  blowing up the first line, the obtained threefold is denoted by $X_1.$ Let $E_1,E_2$ be the
exceptional divisors, then $0\rw T_{E_1}\rw T_{X_1}\rw N_{E_1}X_1\rw 0$ and  $0\rw T_{E_2}\rw T_X\rw N_{E_2}X_1\rw 0.$ Hence $c_1(T_X)\leq c_1(T_{E_2})\leq c_1(T_{X_{1}})\leq c_1(T_{E_1})\leq c_1(T_{P^3})=4$).
 But by the exact sequence $0\rw T_{C_1}\rw T_X|_{C_1}\rw N_{C_1}X\rw 0,$ we have $c_1(T_X|_{C_1})=2+c_1(N_{C_1}X)\geq 32,$ it is a contradiction. Hence the case $x_+<-2$ is impossible.

  In the following we assume $x_+>-2.$

  If $x_+=-1$ then $-K^3_X=36-2y_1.$ If $y_1\geq 3$ then by the classifications in \cite[Table \S 12.5]{ip} and \cite{mm2}, we may have $-K^3_X=26,30$
and $X$ is the blow-up of $P^1\times P^1\times P^1$ along a curve of tridegree of $(1,1,2)$ or $(1,1,3),$ neither of them are possible just for the similar reasoning as the proceeding paragraph:$c_1(T_X|_{C_1})=2+c_1(N_{C_1}X)=2+2x_1y_1= 8,$ but $c_1(T_X|_{C_1})\leq c_1(T_{P^1\times P^1\times P1}|_{C_1})\leq 6.$ Hence $-K^3_X=36,34,32.$

Now we assume $x_+>-1.$ Since $\frac{-4x^2_++2x_+}{2x_++4}=-2x_++5-\frac{20}{2x_++4},$
 by (\ref{dp9}) we have
\begin{eqnarray}2x_+-1\leq 2x_++5-\frac{24}{2x_++4}\leq y_1\leq 2x_++5-\frac{4}{2x_++4},\label{dp10}\end{eqnarray}
  since $y_1$ is an integer we have $2x_+-1\leq y_1\leq 2x_++4.$
  If $y_1=2x_+-1$ then $ -K_X^3=46-8x^2_+-4x_+$ since $x_+\geq 0$ we have  $ -K_X^3=46,34$ for $x_+=0,1.$
  If $y_1=2x_++k$ for $0\leq k\leq 4$ then $$\begin{array}{rcl} -K_X^3&=&42-4x^2_++2x_+(1-2x_+-k)-4(2x_+ +k)\\
  &\leq & 42-8x_+^2-6x_+\\
   \end{array}$$
 and it will be less than $42-32-12=-2$ if $x_+\geq 2,$ if $x_+=1$ then $-K^3_X=28-6k.$ Hence by the classifications in \cite[Table \S 12.5]{ip} and \cite{mm2}, we must have $x_+=0$ and $$-K_X^3=42-4x^2_++2x_+(1-2x_+-k)-4(2x_+ +k)=42-4k,$$
 hence $-K_X^3=42,38,34,30$ for $y_1=2x_+,2x_++1,2x_++2,2x_++3$ respectively.

 To sum up, if $h^{1,2}=0$ then the possible values of $-K_X^3$ are $ 36,34,32,46,42,38,30.$  We will exclude
  $-K_X^3=42$ and $46.$ If  $-K_X^3=32,$ it will have two possibilities, we will exclude one case of them.

If $-K^3_X=32,$ then $X$ has two possibilities by the classifications in \cite[Table \S 12.5]{ip}. In the first case $X$ is the blow-up of $Y$ along a proper transformation of a conic passing through two point $p,q$ and $Y$ is the blow-up of a quadric $Q$ at $p,q.$
In the second case, $X$ is the blow-up of $P^1\times P^2$ along two curves of bidegree of $(2,1)$ and $(1,0)$-respectively.  The first case is impossible since otherwise
we will get a semifree $\C^*$ action on the quadric $Q$ after some equivariant blow downs, which will contradict to Lemma \ref{qqr}.

If $-K_X^3=46,$  by the classifications of Fano threefolds \cite[Table \S 12.5]{ip}, $X$ is the blow-up of $Y$ along the two exceptional lines of the blowing up $Y\rw P^3,$ and $Y$ is the blow-up of $P^3$ along a line. We claim it is impossible by using Proposition \ref{blowu}.
If both exceptional divisors of the blow up $X\rw Y$ are point-wise fixed, then after blowing down we  get a semifree $\C^*$ action on $Y$ whose fixed point set consists of three rational curves which is impossible by Case 2.
If both exceptional divisors of the blow up $X\rw Y$ are not point-wise fixed,  then the $\C^*$ action on $X$ will have more than $2$  interior fixed curves or at least one interior isolated fixed point.
If one exceptional divisor of the blow up $X\rw Y$ is point-wise fixed and the other is not, then the $\C^*$ action on $X$ will also have more than $2$  interior fixed curves or at least one interior isolated fixed point.

If $-K_X^3=42,$  by the classifications of Fano threefolds \cite[Table \S 12.5]{ip}, $X=P^1\times S_7,$ where $S_7$ is the blow-up of $P^2$ at two distinct points. It is clear impossible since there is no semifree  $\C^*$ action on $S_7$ with an isolated fixed point and two fixed lines.

The cases $-K_X^3=30$ and $34$ are impossible. We take  $-K_X^3=30$ as an example and the case $-K_X^3=34$ is proved in the same way.
If  $-K_X^3=30$ then the possible $X$ is the blow-up of $(P^1)^3$
along a curve $C$ of tridegree $(1,1,2).$  The blow up center curve $C$ must be point-wise fixed. If $C$ is not point-wise fixed, then the restricted
 action of $C^*$ on $C$ have isolated fixed points, the fixed point set of the induced $\C^*$ action on $(P^1)^*$ has exactly two connected components and both of them are surfaces isomorphic to $P^1\times P^1,$ and $C$ is a curve connected these two surfaces (otherwise the $\C^*$-action on $X$ has interior isolated fixed points) since the moment map is strictly increasing along the $\C^*$-orbits. Since the $\C^*$-action on $X$ has no interior isolated fixed points, by Proposition \ref{blowu}, the $\C^*$-action on $X$ has at least two interior fixed curves, a contradiction.
 Hence $C$ is  point-wise fixed. Moreover after blow-down we get a semifree $\C^*$-action on   $(P^1)^3$ whose fixed points are exactly two surface both isomorphic to $P^1\times P^1$ and $C$ is a curve located in one of these fixed surfaces. However it will contradict to the fact that $C$ is curve of tridegree $(1,1,2).$

To sum up in the Case 7 we have only the following possibilities:
  \begin{enumerate}
\renewcommand{\labelenumi}{(\arabic{enumi})}
\item $-K_X^3=24,$ the Hodge number $h^{1,1}=4$ and $h^{1,2}=1.$ $X$ is a divisor of multidegree $(1,1,1,1)$ in $P^1\times P^1\times P^1\times P^1.$
\item $-K_X^3=38,$ the Hodge number $h^{1,1}=4$ and $h^{1,2}=0.$ $X$ is the blow-up of $(P^1)^3$
along a curve of tridegree $(0,1,1).$
\item $-K_X^3=32,$ the Hodge number $h^{1,1}=4$ and $h^{1,2}=0.$ $X$ is the blow-up of  $P^1\times P^2$ along two curves of bidegree $(2,1)$ ans $(1,0).$
\item $-K_X^3=36,$ the Hodge number $h^{1,1}=5$ and $h^{1,2}=0.$ $X$ is the blow-up of  $Proj(T_{P^2})$ along two lines.
\end{enumerate}

{\it Case 8.} Both $F_+ $ and $F_-$ are
 $\widetilde{P}^2.$  Note the Marsden-Weinstein reduced space $X_{red}=\widetilde{P}^2$ and  $h^{1,1}=3+a.$ Let $u,v\in H^2(\widetilde{P}^2,\Z)$ be a basis which
are dual classes of the fibres and the zero section $P^1$
respectively. Then
$\int_{\widetilde{P^2}}u^2=0,~~\int_{\widetilde{P^2}}v^2=\int_{P^1}
v=-1, $ and $\int_{\widetilde{P^2}}uv=1.$
Assume $[C_i]=x_i u +y_i v,$ where  $x_i\geq 0$ and $y_i\geq 0$ and $x_i+y_i>0.$
  Note $c_1(F_{\pm})=3u+2v.$ So $c_1(F_{\pm})c_1(N_{F_{\pm}})=2x_{\pm}+y_{\pm}$ and $c^2_1(N_{F_{\pm}})=2x_{\pm}y_{\pm}-y^2_{\pm}.$
By the adjunction formula for curve $C_i$ we have $(-c(F_+)+[C_i])[C_i]+2\chi(C_i)=0,$ i.e.,
\begin{eqnarray}2x_iy_i-y^2_i-2x_i -y_i+2\chi (C_i)=0.\label{dd5}\end{eqnarray}
 Note $c_1^2(F_+)=c_1^2(F_-)=8,$ by (\ref{can3}) we have
\begin{eqnarray}-K_X^3=48 +(2x_+y_+-y_+^2+2x_-y_--y_-^2)+3(2x_+ +y_++2x_- +y_- ).\label{dq4}\end{eqnarray}
By (d) of (\ref{surff}) we have
\begin{eqnarray}2x_+ +y_++2x_- +y_-=-\sum_{i=1}^a (\alpha^+_i+\alpha^-_i+2\chi(C_i)),\label{dq1}\end{eqnarray}
since $2x_+ +y_++2x_- +y_-=2(2x_+ +y_+)-\sum_{i=1}^a(2x_i+y_i)$ and $\alpha^+_i+\alpha^-_i=2x_iy_i-y^2_i,$ we have $2x_+ +y_+=0.$

{\it Subcase 1.} $a=1$ and $h^{1,2}=1.$ By the classifications of Fano threefolds \cite[Table \S 12.5]{ip}, the Fano threefolds with $h^{1,1}=4$ must have $-K^3_X\geq 24.$  Then by  (\ref{dd5}) we have $x_1=\frac{y_1(y_1+1)}{2(y_1-1)},$ since $x_1\geq 0,y_1\geq 0 $ and $x_1+y_1\geq 1$ we must have
$y_1\geq 2.$  If $y_1\geq 4$  then $x_1=1+\frac{y_1}{2}+\frac{1}{y_1-1}$ is not an integer. Hence $y_1=2$ or $y_1=3$ and $(x_1,y_1)=(3,2)$ or $(3,3).$
Hence $\alpha^+_1+\alpha^-_1=2x_1 y_1-y_1^2=8$ or $9.$ By (d) of (\ref{surff}) we have
\begin{eqnarray}2x_+ +y_++2x_- +y_-=-(\alpha^+_1+\alpha^-_1).\label{dq1}\end{eqnarray}
If $(x_1,y_1)=(3,2),$  use  (\ref{dq4}) we get
$$
-K_X^3=24+(-8x^2_+ +2(x_+-3)(y_+-2)-(y_+-2)^2)
=32-16 x^2_+ $$
hence  $x_+=0$ and $-K_X^3=32.$
If $(x_1,y_1)=(3,3),$ then
$$
-K_X^3=21+(-8x^2_+ +2(x_+-3)(-2x_+-3)-(-2x_+-3)^2)
=30-16 x^2_+-6x_+,
$$
therefore   $x_+=0$ or $\pm 1$ and $-K_X^3=30$ or $18.$

However, by the classifications of Fano threefolds \cite[Table \S 12.5]{ip}, the Fano threefolds $X$ with $h^{1,1}=4$ and $h^{1,2}=1$ must have $-K^3_X=24$ or $28,$ hence this subcase can't happen.

{\it Subcase 2.} $a=1$ and $h^{1,2}=0.$ Then by  (\ref{dd5}) we have $x_1=\frac{y_1(y_1+1)+2}{2(y_1-1)},$ since $x_1\geq 0,y_1\geq 0 $ and $x_1+y_1\geq 1$ we must have
$y_1\geq 2.$  If $y_1\geq 6$  then $x_1=1+\frac{y_1}{2}+\frac{2}{y_1-1}$ is not an integer. Hence $2\leq y_1\leq 5.$  In order that both $x_1$ and $y_1$ are integers, it is easy to check that $(x_1,y_1)=(4,2)$ and $(x_1,y_1)=(4,5)$ are only possibilities.
If $(x_1,y_1)=(4,2)$ then by (d) of (\ref{surff}) we have
\begin{eqnarray}2x_+ +y_++2x_- +y_- = -(\alpha^+_1+\alpha^-_1+2\chi (C_1))=-14.\label{dq6}\end{eqnarray}
By (\ref{dq4}) and (\ref{dq6}) we have
$$
-K_X^3=6-8x_+^2- 2(x_+-4)(2x_++2)-(2x_++2)^2
=18-16 x^2_+ +4x_+$$
Hence   $-K_X^3\leq 22.$
If $(x_1,y_1)=(4,5)$ then then By (d) of (\ref{surff}) we have
\begin{eqnarray}2x_+ +y_++2x_- +y_- = -(\alpha^+_1+\alpha^-_1+2\chi (C_1))=-17.\label{dq6}\end{eqnarray}
By (\ref{dq4}) and (\ref{dq6}) we have
$$\begin{array}{rcl}
-K_X^3&=&-3-8x^2_+ -2(x_+-4)(2x_++5)-(2x_++5)^2\\
&=&12-16 x^2_+ -14x_+,\end{array}$$
hence $-K_X^3\leq 12.$ However, by the classifications of Fano threefolds \cite[Table \S 12.5]{ip} and \cite{mm2}, the Fano threefolds $X$ with $h^{1,1}=4$ and $h^{1,2}=0$ must have $-K^3_X\geq 26,$ hence this subcase can't happen.

{\it Subcase 3.} $a=2.$ In this case we must have $h^{1,1}=5$ and $h^{1,2}=0.$ Just as Subcase 2, we must have $(x_i,y_i)=(4,2)$ or $(4,5).$
Assume there are $k$-pairs $(4,2)$ and $2-k$ pairs $(4,5).$ Then
\begin{eqnarray}2x_+ +y_++2x_- +y_- = -14k-17(2-k)=3k-34.\label{dq6}\end{eqnarray}
Note $x_-=x_+-8$   and $y_-=-2x_+-2k-(2-k)5=-2x_++3k-10.$ Hence
$$\begin{array}{rcl}
-K_X^3&=&48+3(3k-34)-8x^2_+ -2(x_+-8)(2x_+-3k+10)-(2x_+-3k+10)^2\\
&=&48+3(3k-34)-16(3k-10)-(3k-10)^2-16 x^2_+ +(18k-28)x_+\\
&=&6+21k-9k^2-16 x^2_+ +(18k-28)x_+.
\end{array}$$
 By the classifications of Fano threefolds with $h^{1,1}=5$ in \cite[Table \S 12.6]{ip},  we must have  $ -K^3_X=28, 36.$
However, if $k=0$  then $-K^3_X=6-16 x^2_+ -28x_+,$ hence we must have $x_+=-1$ and $-K^3_X=34.$
If $k=1$  then $-K^3_X=18-16 x^2_+ -10x_+,$ hence we must $-K^3_X\leq 12.$
If $k=2$  then $-K^3_X=12-16 x^2_+ +8x_+,$ hence we must  $-K^3_X\leq 12.$
Hence the Subcase 3 is impossible by the classifications of Fano threefolds \cite[Table \S 12.5]{ip} and \cite{mm2}.

{\it Subcase 4.} $a\geq 3.$ By the classifications of Fano threefolds in \cite[Table \S 12.6]{ip}, $X$ is of form $P^1\times S_k$ with $k\leq 5,$
where $S_k$ is the blow up of $P^2$ at $9-k$ distinct general points. Since the exceptional lines are invariant under $\C^*$-action, it means the fixed point set $X^{\C^*}$ have more than $9-5=4$ connected components, hence this subcase is also impossible.

To sum up, the case 8 is impossible.

\end{proof}

\noindent{\bf Acknowledgments}~~Part of this work was done
while the first named author's stay at Mathematical Department of Harvard University
 during 2014 and the second named author's visit at Mathematical Department of Sun Yat-Sen University during
2012. The authors are grateful to both departments for generous supports and hospitalities. They also would like to thank Chen, Dawei at Boston University and Li, Hui at Suzhou University for answering their questions.


\begin{thebibliography}{kuanuda}

\bibitem[At82]{at82}{M. F. Atiyah, Convexity and commuting Hamitonians,}
 {\it Bull. London Math. Soci.,} {\bf 14} {(1982), 1--15.}

\bibitem[AB67]{ab}{M. F. Atiyah \& R. Bott, A Lefschetz fixed point formula for elliptic complexes, I and II.,}
 {\it Ann. Math.,} {\bf 86} {(1967), 374--407, {\bf 88} (1968), 451--491.}

\bibitem[AB84]{ab84}{M. F. Atiyah \& R. Bott, The moment map and equivariant cohomology, {\it Topology,} {\bf 23} (1984),  1--28.}

\bibitem[AS68]{as}{M. F. Atiyah \& I. M. Singer, The index of elliptic operators, III.,}
 {\it Ann. Math.,} {\bf 87} {(1968), 546--604.}

\bibitem[Au91]{au} {M. Audin,} {\it The topology of torus actions on symplectic
manifolds.} {Translated from the French by the author. Progress in
Mathematics, {\bf 93}, Birkh\"auser Verlag, Basel, 1991.}

\bibitem[BB73]{bb}{ A. Bialynicki-Birula, Some theorems on actions of
algebraic groups,} {\it Ann. of Math.,} {\bf 98}{ (1973), 480--497.}

\bibitem[C92]{cm}{F. Campana, Connexit\'e rationnelle des vari\'et\'es de Fano, {\it Ann. Sci. \'Ecole Norm. Sup.}, {\bf 25} (1992),  539--545. }

\bibitem[CS79]{cs79}{J. B. Carrell \& A. J. Sommese,  Some topological aspects of $\C^*$-actions on
compact K\"ahler manifolds,}{ {\it Comment.
Math. Helv.,} {\bf 54} (1979),  567--3582}

\bibitem[De88]{del}{ T. Delzant,  Hamiltoniens p\'eriodiques et
images convexes de l'application moment,} {\it Bull. Soc. Math.
France,} {\bf 116} {(1988), 315--339.}

\bibitem[Fr59]{fr59}{T. Frankel, Theodore Fixed points and torsion on K\"ahler manifolds,} {\it Ann. of Math.,} {\bf 70} {(1959) 1--8.}

\bibitem[Fu79]{fu79}{A. Fujiki,  Fixed points of the actions on compact K\"ahler
manifolds,} {{\it Publ. Res. Inst. Math. Sci.,} {\bf 15} (1979),
797--826.}

\bibitem[Ft87]{fut}{A. Futaki, The Ricci curvature of symplectic quotients of Fano
manifolds,}{ {\it Tohoku Math. J.,} {\bf 39} (1987),  329--339.}

\bibitem[GS89]{gs89}{ V. Guillemin \& S. Sternberg, Birational equivalence in the symplectic category,}
 {\it Invent. Math.,} {\bf 97}{ (1989), 489-522.}



\bibitem[HL94]{hl94}{P. Heinzner \& F. Loose,
Reduction of complex Hamiltonian $G$-spaces,
{\it Geom. Funct. Anal.,} {\bf 4} (1994),  288¨C-297.}

\bibitem[Hir64]{hir64}{H. Hironaka,  Resolution of singularities of an algebraic variety over a field of characteristic zero. I, II.,} {\it Ann. of Math.,} {\bf 79} {(1964), 109--326.}

\bibitem[IP99]{ip}{V. A. Iskovskih \& Yu. G. Prokhorov,  Fano varieties, {\it Algebraic geometry,} {\bf V}, 1--247, Encyclopaedia Math. Sci., 47, Springer, Berlin, 1999.}

\bibitem[Ka99]{ka}{Y. Karshon,  Periodic Hamiltonian flows on four-dimensional manifolds, {\it Mem. Amer. Math. Soc.,} 141 (1999), no. {\bf 672}.}

\bibitem[KMM92]{kmm92}{J. Koll\'ar and Y. Miyaoka and S. Mori,  Rationally connected varieties,} {\it J. Algebraic Geom.,} {\bf 1} {(1992),  429-¨C448.}


\bibitem[Li03]{lh}{H. Li, Semi-free Hamiltonian circle actions on 6-dimensional
symplectic manifolds,} {\it Trans. Amer. Math. Soc.,} {\bf 355}
{(2003), 4543--4568.}

\bibitem[MM81]{mm1}{S. Mori \& S. Mukai, Classification of Fano $3$-folds with $B_2
\geq 2$,} {\it Manuscripta Math.,} {\bf 36} {(1981/82), 147--162.}

\bibitem[MM86]{mm}{S. Mori \& S. Mukai, Classification of Fano $3$-folds with $B_2
\geq 2,$ I., {\it Algebraic and topological theories} (Kinosaki, 1984), 496--545, Kinokuniya, Tokyo, 1986.}

\bibitem[MM03]{mm2}{S. Mori \& S. Mukai, Erratum: Classification of Fano $3$-folds with $B_2
\geq 2$,} {\it Manuscripta Math.,} {\bf 110} {(2003), 407.}

\bibitem[MU83]{mu}{S. Mukai \& S. Mukai \& H. Umemura,
Minimal rational threefolds, {\it Algebraic geometry}} {(Tokyo/Kyoto, 1982), 490--518,
Lecture Notes in Math.,} {\bf 1016}, {Springer, Berlin, 1983.}

\bibitem[OW77]{ow77}{P.Orlik and P. Wagreich,  Algebraic surfaces with $k^*$-action,}  {\it Acta Math.,} {\bf 138} {(1977),  43--81.}

\bibitem[Sum74]{sum74}{H. Sumihiro, Equivariant completion,}
{\it J. Math. Kyoto Univ.,} {\bf 14} {(1974), 1--28}

\bibitem[TW00]{tw00}{S. Tolman \& J. Weitsman, On semifree symplectic circle actions with isolated fixed points,}
 {\it Topology,} {\bf 39} {(2000), 299--309.}

\bibitem[Wa66]{wall}{ C.T.C. Wall, Classification problems in differential topology V: On certain 6-manifolds,}
{\it Invent. Math.,} {\bf 1} {(1966),355-374}


\bibitem[Y08]{yb}{ Q.-L. Yang, Morse and semistable stratifications of K\"ahler spaces
by $\mathbb C^*$-action,} {\it Monatsh. Math.,} {\bf 155} {(2008),
79--95.}


\end{thebibliography}
 \end{document}